\documentclass[10pt,a4paper]{amsart}
\usepackage{hyperref}
\usepackage{mathrsfs}
\usepackage{amsmath}
\usepackage{amssymb}
\usepackage[all]{xy}
\usepackage{latexsym}
\usepackage[T1]{fontenc}
\title[Connective $K$-theory]{On connective $K$-theory of elementary abelian $2$-groups and local duality}
\author[Geoffrey Powell]{Geoffrey M.L. Powell}
\address{Laboratoire Analyse, Géométrie et Applications, UMR 7539\\ Institut
Galilée, Université Paris 13, 93430 Villetaneuse, France}
\email{powell@math.univ-paris13.fr}
\thanks{This work was partly financed by the project  ANR BLAN08-2 338236, HGRT}
\date{}
\newtheorem{thm}{Theorem}[section]
\newtheorem{prop}[thm]{Proposition}
\newtheorem{cor}[thm]{Corollary}
\newtheorem{lem}[thm]{Lemma}

\theoremstyle{definition}
\newtheorem{defn}[thm]{Definition}
\newtheorem{exam}[thm]{Example}

\theoremstyle{remark}
\newtheorem{rem}[thm]{Remark}
\newtheorem{nota}[thm]{Notation}

\newcommand{\hz}{H \zed}
\newcommand{\dbicom}{\mathbf{D}}
\newcommand{\bicom}{\mathbf{B}}
\newcommand{\kz}{\mathbf{Kz}}
\newcommand{\smodf}{S^\bullet\dash\modules_\f}
\newcommand{\smodhom}[1][V]{\hom_{S^\bullet}^{#1}}
\newcommand{\smodaut}[1][V]{S^\bullet(V)\dash\modules_{\aut(#1)}}

\newcommand{\lfrak}{\mathfrak{L}}
\newcommand{\kfrak}{\mathfrak{K}}
\newcommand{\torsion}{\mathbf{tors}_v}
\newcommand{\cotorsion}{\mathbf{cotors}_v}
\newcommand{\annih}{\mathbf{ann}_v}
\newcommand{\qfrak}{\mathfrak{Q}}

\newcommand{\Pzed}{P_{\zed}}
\newcommand{\Pzedtwo}{P_{\zed_2}}
\newcommand{\Pbarzed}{\overline{P}_\zed}
\newcommand{\Pbarzedtwo}{\overline{P}_{\zed_2}}
\newcommand{{\ab}}{\mathscr{A}b}

\newcommand{\vs}{\mathscr{V}}
\newcommand{\fdvs}{{\mathscr{V}^f}}
\newcommand{\f}{\mathscr{F}}
\newcommand{\ff}{\mathscr{FA}}

\newcommand{\op}{^{\mathrm{op}}}
\newcommand{\obj}{\mathrm{Ob}\hspace{1pt}}
\newcommand{\ext}{\mathrm{Ext}}

\newcommand{\aut}{\mathrm{Aut}}
\newcommand{\nat}{\mathbb{N}}
\newcommand{\modules}{\mathbf{mod}}
\renewcommand{\hom}{\mathrm{Hom}}
\renewcommand{\phi}{\varphi}
\renewcommand{\epsilon}{\varepsilon}
\newcommand{\symm}{\mathfrak{S}}
\newcommand{\zed}{\mathbb{Z}}
\newcommand{\zedloc}[1][p]{\zed_{(p)}}
\newcommand{\dash}{\hspace{-2pt}-\hspace{-2pt}}
\newcommand{\field}{\mathbb{F}}
\newcommand{\cala}{\mathcal{A}}
\newcommand{\Ibar}{\overline{I}_\field}
\newcommand{\Pbar}{\overline{P}_\field}
\begin{document}

\begin{abstract}The connective $ku$-(co)homology of elementary abelian $2$-groups is determined as a functor of the elementary abelian $2$-group. The argument requires only the calculation of the rank one case and the Atiyah-Segal theorem for $KU$-cohomology together with an analysis of the functorial structure of the integral group ring. The methods can also be applied to the odd primary case.

These results are used to analyse the local cohomology spectral sequence calculating $ku$-homology, via a functorial version of local duality for Koszul complexes. This gives a conceptual explanation of results of Bruner and Greenlees.
\end{abstract}
\maketitle

\section{Introduction}

The calculation of the  $ku$-(co)homology of finite groups is an
interesting and highly non-trivial problem. The case of elementary
abelian $p$-groups illustrates important features; these groups were first
calculated by Ossa \cite{ossa} and were studied further by Bruner and Greenlees 
\cite{bruner_greenlees}, exhibiting a form of duality via local cohomology.
Neither of these references exploit the full naturality of the functors $V \mapsto ku^* (BV_+) $
and $V  \mapsto ku_* (BV_+)$.

This paper shows how studying these as functors of the elementary abelian $p$-group $V$ gives a new and conceptual approach. The methods apply to any
prime $p$; the case $p=2$ is privileged here since this requires an additional filtration argument when studying the local cohomology. Moreover, this is the case of interest when extending the methods to the study of $ko$-(co)homology (cf. \cite{bg2}); this will be developed elsewhere.  The main results of the first part of the paper give complete descriptions of the functors $V \mapsto ku^* (BV_+) $ (Theorem \ref{thm:ku_cohom_BV}) and $V\mapsto ku_* (BV_+)$ (Theorem \ref{thm:ku_homology_BV}).

For $ku$-cohomology, the only input which is required is the graded abelian group structure of  $ku^* (B\zed/2_+)$  and the identification of the functor $V \mapsto KU^0 (BV_+) $ for periodic complex $K$-theory, which is provided by the Atiyah-Segal completion theorem. In particular, the method gives a conceptual proof of an algebraic form of Ossa's theorem \cite{ossa}, which gives a reduction to the rank one case. Both Ossa \cite{ossa} and  Bruner and Greenlees \cite{bruner_greenlees} use Ossa's theorem as a starting point for their calculations.

To give a full functorial description,  the integral group ring functor $V \mapsto \zed[V]$ is studied, stressing the functorial viewpoint and extending results of Passi and others \cite{passi_vermani}.  Of independent interest is the observation that the quotients which arise from studying the filtration of $\zed[V]$ by powers of the augmentation ideal are self-dual under Pontrjagin duality (see Theorems \ref{thm:pontrjagin_duality} and \ref{thm:Q-self-dual}).

Similarly, for $ku$-homology,  only  knowledge  of $ku_* (B\zed/2_+)$ is required, together with an understanding of the functor $V \mapsto KU_1 (BV_+)$. The arguments make explicit the functorial nature of the duality between $ku^*(BV_+)$ and $ku_* (BV_+)$, via  Pontrjagin duality.

The second part of the paper applies these results to give an analysis of the local cohomology spectral sequence relating $ku^* (BV_+)$ to $ku_* (BV_+)$ (see Theorem \ref{thm:local_cohom_ss}); this sheds light upon the description given by  Bruner and Greenlees \cite{bruner_greenlees}: local duality appears as an explicit functor defined in the functorial context. The key observation which explains the origin of the differentials in the local cohomology spectral sequence comes from the analysis of $ku^* (BV_+)$, which shows how the $v$-torsion $\torsion ku^*(BV_+)$ and the $v$-cotorsion of  $ku^* (BV_+)$  are related.

 The functorial description of $V \mapsto ku^* (BV_+) $ identifies the mod-$p$ cohomology of the spaces of the $\Omega$-spectrum for $ku$ (up to nilpotent unstable modules), via Lannes' theory \cite{lannes}. This gives a conceptual framework for understanding the results of \cite{stong,singer}, and can be related to the description of the mod-$p$ homology in terms of Hopf rings \cite{hara} (which does not {\em a priori} retain information on the action of the Steenrod algebra). This will be explained elsewhere.

\tableofcontents
\section{Background}

\subsection{Definitions and notation}

Fix  a prime $p$ and let   $\field$ denote the prime field $\field_p$;  $\fdvs$ denotes 
the full subcategory of finite-dimensional spaces in the category $\vs$ of $\field$-vector spaces and  vector space duality is denoted by $(-)^\sharp : \vs \op \rightarrow \vs$. 

\begin{nota}
The category of functors from  $\fdvs$ to abelian groups is denoted $\ff$ and the full subcategory of functors
with values in $\vs$ is denoted $\f$. 
\end{nota}

The categories $\ff, \f$ are tensor abelian, with structure induced from $\ab$. (For
basic properties of $\f$, see \cite{kuhn1,kuhn2,kuhn3} or \cite{ffss}.) There is an exact Pontrjagin duality functor which generalizes the duality for $\f$ introduced in \cite{kuhn1}:

\begin{defn}
\label{def:Pontrjagin}
 Let $D : \ff \op \rightarrow \ff$ be the functor defined on $F \in \ff$ by 
\[
 D F (V ) : = \hom_{\ab} (F (V^{\sharp}), \zed /p^\infty).
\]
\end{defn}

Recall that the socle of an object is its largest semi-simple subobject and the head its largest semi-simple quotient.

\begin{exam}
The symmetric powers, divided powers and exterior powers are fundamental examples of (polynomial) functors in $\f$.
For $n \in \nat$, the $n$th symmetric power functor $S^n$ is defined by $S^n (V) := (V^{\otimes n})/\symm_n$, the $n$th divided power functor by $ \Gamma^n (V) := (V^{\otimes n})^{\symm_n}$ and the $n$th exterior power functor identifies as $\Lambda^n (V) \cong (V^{\otimes n } \otimes \mathrm{sign})^{\symm_n}$, where $\mathrm{sign}$ is the sign representation of $\symm_n$. By convention, these functors are  zero for negative integers $n$. There is a duality relation $S^n \cong D \Gamma^n$, whereas the functor $\Lambda^n$ is self-dual (there is a canonical isomorphism $D \Lambda^n \cong \Lambda^n$). For $p=2$, the functor $\Lambda^n$ is the head of $S^n$ and the socle of $\Gamma^n$. 

These functors are examples of graded exponential functors; for example, the exponential structure  induces natural coproducts $S^n \stackrel{\Delta}{\rightarrow} S^ i\otimes S^j $ and products $S^i \otimes S^j \stackrel{\mu}{\rightarrow} S^n$, for
integers $i+j =n$, and, upon evaluation on $V \in \obj \fdvs$, these correspond to the  primitively-generated Hopf algebra structure on a polynomial algebra.
\end{exam}

\begin{exam}
\label{exam:standard_injective}
Yoneda's lemma provides the standard injective and projective objects of $\f$ (the case of $\ff$ is considered in Section \ref{sect:groupring}).  The projective functor $P_\field$ is the functor $V \mapsto \field [V]$, which corepresents  evaluation at $\field$; the injective functor $I_\field$ is given by $V \mapsto \field^{V^\sharp}$ (the
vector space of set maps) and represents the dual evaluation functor $F \mapsto DF(\field)$. Duality provides the relation $I_\field \cong D P_\field$, which relates the  canonical decompositions 
$I_\field \cong \field \oplus \Ibar$ and $P_\field \cong \field \oplus \Pbar$, where $\field$
is the constant functor and $\Ibar$ (respectively $\Pbar$) is the complementary constant-free summand.

The functor $I_\field$ is ungraded exponential and has associated diagonal $\Delta : I_\field \rightarrow I_\field
\otimes I_\field$ and multiplication $\mu : I_\field \otimes I_\field \rightarrow I_\field$; these morphisms induce 
$\Ibar \rightarrow \Ibar \otimes \Ibar $ and $\Ibar \otimes \Ibar \rightarrow \Ibar$ respectively and, in both cases,
these are the unique non-trivial morphisms of the given form. Dually, there is a product $\Pbar \otimes \Pbar
\rightarrow \Pbar $ and  coproduct $\Pbar \rightarrow \Pbar \otimes \Pbar$. 
\end{exam}

\begin{nota}
 For $n >0$ an integer:
\begin{enumerate}
 \item 
let $\Pbar^n$ be the image of the iterated product $\mu^{(n-1)} : \Pbar ^{\otimes n}
\rightarrow \Pbar$ and $q_{n-1} \Pbar$ denote its cokernel, so that there is a short exact sequence 
\[
 0
\rightarrow 
\Pbar^n \rightarrow
\Pbar
\rightarrow
q_{n-1} 
\Pbar
\rightarrow 
0;
\]
\item
dually, let $p_{n-1} \Ibar$ denote the kernel of the iterated diagonal $\Ibar \stackrel{\Delta^{(n)}}{\rightarrow} \Ibar ^{\otimes n}$.
\end{enumerate}
\end{nota}

\begin{lem}
 \label{lem:p=2_inj_proj}
\cite{kuhn2}
Suppose that $p=2$. Then 
\begin{enumerate}
 \item 
$\Ibar$ is the injective envelope of $\Lambda^1$, is uniserial,  $\Ibar \cong  \mathrm{colim}_{\rightarrow}  p_n \Ibar$,  and there are non-split short exact sequences 
\[
 0 
\rightarrow p_n \Ibar 
\rightarrow
p_{n+1}\Ibar 
\rightarrow 
\Lambda^{n+1}
\rightarrow 0;
\]
\item
$\Pbar$ is the projective cover of $\Lambda^1$,  is uniserial, $ \Pbar \cong \lim_\leftarrow q_n \Pbar$,  and there are short exact sequences 
\[
 0 
\rightarrow
\Lambda^{n+1}
\rightarrow
q_{n+1} \Pbar 
\rightarrow 
q_n \Pbar 
\rightarrow 
0.
\]
\end{enumerate}
The functor $q_n\Pbar$ is the dual of $ p_{n} \Ibar$. 
\end{lem}

The fact that $p_n \Ibar$ has a simple socle (for $n >0$) implies that it is easy to detect non-triviality of a subobject:
 
\begin{lem}
\label{lem:subfunctor_pdI_qdP}
Suppose that $p=2$ and  let $n>0$ be an integer.
\begin{enumerate}
 \item 
If $G \subset p_n \Ibar$,  then $G =0$ if and only if $G(\field) =0$. 
\item
If $H \subset q_n \Pbar$,  then $H= q_n \Pbar$ if and only if $H (\field)
\neq 0$.
\end{enumerate}
\end{lem}

\begin{proof}
The two statements are equivalent by duality, hence it suffices to prove the first. 
The socle of $p_n \Ibar $ is $\Lambda^1$; if $G \neq 0$ then $\Lambda^1 \subset G$, hence $G(\field) \neq 0$, since
$\Lambda ^1 (\field) = \field$. The
converse is obvious. 
\end{proof}

\begin{rem}
 There are analogous  results  for odd primes, taking into account the weight splitting of the category $\f$ provided by the action of the units $\field_p^\times$ (cf. \cite{kuhn1}).
\end{rem}

\section{The integral group ring functor}
\label{sect:groupring}

This section provides a functorial analysis of the structure of the integral group ring functor;  throughout, the prime is taken to be $2$ (there are analogous results for odd primes). These results are necessary to complete the full functorial description of the $ku$-(co)homology of elementary abelian $2$-groups but are not required for the proof of the algebraic version of Ossa's theorem.

\subsection{The functors $\Pbarzed$, $\Pbarzedtwo$}
\label{subsect:Pbarzed}

\begin{nota} Let 
\begin{enumerate}
 \item 
$\Pzed$ denote the integral group ring functor $V \mapsto \zed [V]$ and $\Pbarzed$  the augmentation ideal, so that there is a direct sum decomposition $\Pzed \cong \zed \oplus \Pbarzed$ in $\ff$;
\item
$\Pzedtwo$ denote the functor $\zed_2 \otimes \Pzed$ and $\Pbarzedtwo$ the functor $\zed_2 \otimes \Pbarzed$, where $\zed_2$ denotes the $2$-adic integers.
\end{enumerate}
\end{nota}

Yoneda's lemma implies: 

\begin{lem}
\label{lem:yoneda}
 The functor $\Pzed$ is projective in $\ff$ and corepresents  evaluation on $\field$.
\end{lem}

The ring structure of $\zed [V]$ gives a morphism 
$\mu : \Pzed \otimes \Pzed \rightarrow \Pzed$, 
which induces $\mu : \Pbarzed \otimes \Pbarzed \rightarrow \Pbarzed$. There is a reduced diagonal $\Delta : \Pbarzed \rightarrow \Pbarzed \otimes \Pbarzed$;  the composition with the canonical projection $\Pzed \twoheadrightarrow \Pbarzed$ is determined by the element $([1] - [0]) \otimes ([1]- [0]) \in ( \Pbarzed \otimes \Pbarzed )(\field)$, by Yoneda.

\begin{defn}
 For $n\in \nat $, let $\Pbarzed^n$ denote the image of the iterated product
 $
 \mu^{(n-1)} : \Pbarzed^{\otimes n} 
\rightarrow 
\Pbarzed
$ 
(respectively  $\Pbarzedtwo^n \subset \Pbarzedtwo$). 
\end{defn}

The structure of $\Pbarzed(V)$ and the  filtration
\[
 \ldots \subset \Pbarzed^{n+1} (V) \subset \Pbarzed^n (V) \subset \ldots \subset \Pbarzed^1(V) = \Pbarzed(V)
\]
for a fixed $V$ has received much attention (see
\cite{passi_vermani,bak_vavilov}, for instance). These references do not exploit functoriality.

\begin{lem}
\label{lem:proj_mod_2}
\ 
\begin{enumerate}
 \item 
There is a natural isomorphism $
 (\Pbarzedtwo )/2 \cong \Pbar.
$
\item
For $n \in \nat$,  the canonical surjection $\Pbarzedtwo \twoheadrightarrow \Pbar$
induces a commutative diagram in $\ff$:
\[
 \xymatrix{
\Pbarzedtwo^n 
\ar@{^(->}[r]
\ar@{->>}[d]
&
\Pbarzedtwo
\ar@{->>}[d]
\\
\Pbar^n 
\ar@{^(->}[r]
&
\Pbar.
}
\]
\item
The head of $\Pbarzedtwo$ is the functor $\Lambda^1$.
\end{enumerate}
\end{lem}

\begin{proof}
The commutative diagram follows from the fact that $\zed  \rightarrow \field$ induces a morphism of group rings $\zed [V] \rightarrow \field [V]$; the remaining statements are clear.
\end{proof}

\begin{lem}
 \label{lem:2_times_filt}
The composite $\Pbarzed \stackrel{\Delta} {\rightarrow }\Pbarzed \otimes \Pbarzed \stackrel{\mu}{\rightarrow} \Pbarzed$
is the morphism $\Pbarzed \stackrel{-2}{\rightarrow} \Pbarzed$. Hence, for  $n \in \nat$, there are inclusions
of subobjects of $\Pbarzed$:
$
 2 \Pbarzed^{n} \subset \Pbarzed^{n+1} 
$
and  $2^n \Pbarzed \subset \Pbarzed^{n+1}$. 

\end{lem}

\begin{proof}
 The first statement is straightforward (cf. \cite[Lemma 3.2]{bak_vavilov}); this gives rise to the natural inclusion $
 2 \Pbarzed^{n} \subset \Pbarzed^{n+1}
$. The final statement follows by induction.
\end{proof}

\begin{lem}
\label{lem:Pbarzed_filt_quot_surject}
For $n \in \nat$, there is a unique non-trivial morphism $\Pbarzed^n
\rightarrow \Ibar$ in $\ff$ and this
induces a surjection 
\[
 \Pbarzed^n /\Pbarzed^{n+1} 
\stackrel{}{\twoheadrightarrow}
p_n \Ibar.
\]
\end{lem}

\begin{proof}
 A morphism $\Pbarzed^n \rightarrow \Ibar$ factors naturally across $(\Pbarzed^n) /2$. It is straightforward to see
that  $\big((\Pbarzed^n) /2\big)( \field) = \field$ and $\big((\Pbarzed^n) /2\big)( 0) = 0$; this implies that there is
a unique non-trivial morphism $(\Pbarzed^n) /2 \rightarrow \Ibar$, by Yoneda.
 
The composite $\Pbarzed^{\otimes n} \twoheadrightarrow \Pbarzed^n \rightarrow \Ibar$ factorizes across the projection
$\Pbarzed^{\otimes n} \twoheadrightarrow \Pbar^{\otimes n}$. Again, there is a unique non-trivial morphism from
$\Pbar^{\otimes n}$ to $\Ibar$; for $n=1$, this is the composite $\Pbar \twoheadrightarrow \Lambda^1 \hookrightarrow
\Ibar$ and, for $n>1$, 
\[
 \Pbar^{\otimes n} 
\rightarrow 
\Ibar^{\otimes n} 
\stackrel{\mu^{(n-1)}}{\longrightarrow}\Ibar, 
\]
where the first morphism is the iterated tensor product of the morphism $\Pbar \rightarrow \Ibar$; it follows easily that  the composite  has image $p_n\Ibar$. Hence the morphism $\Pbarzed^n \rightarrow \Ibar$ has image $p_n \Ibar$, by Lemma \ref{lem:p=2_inj_proj}.

Finally, the fact that $\Pbar^2 $ is the kernel of $\Pbar \twoheadrightarrow \Lambda^1$ implies that the composite
\[
 \Pbarzed  ^{n+1} 
\hookrightarrow
\Pbarzed^n 
\rightarrow 
\Ibar
\]
is trivial, giving the stated factorization.
\end{proof}

A non-functorial version of the following result (expressed using different notation) is proved in \cite{passi_vermani}. Since the functorial version is required here, a direct proof is given.

\begin{prop}
\label{prop:subquotients_Pbarzed}
 For $n>0$  an integer, the following  statements hold:
\begin{enumerate}
\item
there is a short exact sequence:
\[
 0 \rightarrow 2 \Pbarzed^{n} \rightarrow \Pbarzed^{n+1} \rightarrow
\Pbar^{n+1} \rightarrow 0;
\] 
\item 
the canonical morphism $\Pbarzed^{n}/\Pbarzed^{n+1} \rightarrow
p_n\Ibar$ is an isomorphism.
\end{enumerate}
\end{prop}

\begin{proof}
 The statements are proved  in parallel  by induction upon $n$. For
 $n=1$, consider the commutative diagram
\[
 \xymatrix{
2 \Pbarzed 
\ar@{=}[d]
\ar@{^(->}[r]
&
\Pbarzed^2 
\ar@{^(->}[d]
\ar@{->>}[r]
&
\Pbarzed^2/2 \Pbarzed
\ar@{^(->}[d]
\\
2 \Pbarzed
\ar@{^(->}[r]
&
\Pbarzed
\ar@{->>}[r]
&
\Pbar.
}
\]
The image of $\Pbarzed^2$ in $\Pbar$ is $\Pbar^2$, by Lemma
\ref{lem:proj_mod_2}, which proves the first statement. The second statement
follows by studying the cokernels of the monomorphisms.

For the inductive step, suppose  the result true for  $n < N$.
Lemma \ref{lem:2_times_filt} provides a natural 
inclusion $2^i \Pbarzed \subset \Pbarzed^{i+1}$, for all natural numbers $i$;
for current purposes, it is sufficient to work with the inclusion $2^{i+1}
\Pbarzed \subset \Pbarzed^{i+1}$. The proof proceeds by providing upper and
lower bounds for $\Pbarzed^{N+1} / 2^{N+1} \Pbarzed$. It is sufficient to do
this in the Grothendieck group of $\ff$, since  the functors considered below only have finitely many composition factors of a given isomorphism type, which allows comparison arguments. (In fact, evaluation on finite dimensional vector spaces allows the arguments to be reduced to objects which are finite.) 

The element of the Grothendieck group associated to an object $F$ is denoted $[F]$; for the functors considered here, this lies in the submonoid $\prod_\lambda
\nat$ indexed by the isomorphism classes of simple objects $S_\lambda$ of $\ff$; for two objects $M = \Sigma
M_\lambda [S_\lambda ]$ and $N = \Sigma N_\lambda [S_\lambda]$, write $M \leq N$ if $M_\lambda \leq
N_\lambda$, for all $\lambda$.

The inductive hypothesis implies that: 
\[
 \big[ 
\Pbarzed^{N} / 2^{N} \Pbarzed
\big] 
= 
\sum _{i=1}^N
\big[
\Pbar^{i}
\big]
\]
in the Grothendieck group.

The inclusion $2\Pbarzed^N \hookrightarrow \Pbarzed^{N+1}$ induces a
monomorphism 
$$
(\Pbarzed^N /2^{N}\Pbarzed) \cong 
2\Pbarzed^N /2^{N+1}\Pbarzed \hookrightarrow
\Pbarzed^{N+1}/2^{N+1} \Pbarzed;$$
 the composite of this morphism with the
projection $\Pbarzed^{N+1} /2^{N+1} \Pbarzed \twoheadrightarrow \Pbar^{N+1}$ is
clearly trivial, hence this gives a lower bound for $\Pbarzed^{N+1} / 2^{N+1}
\Pbarzed$:
\[
\sum _{i=1}^{N+1}
\big[
\Pbar^{i}
\big] 
\leq 
\big[\Pbarzed^{N+1} / 2^{N+1}
\Pbarzed
\big], 
\]
with equality if and only if the first statement holds.

Consider the inclusion $\Pbarzed^{N+1} \hookrightarrow \Pbarzed^N$, which
induces the monomorphism 
\[
 \Pbarzed^{N+1} / 2^{N+1} \Pbarzed \hookrightarrow \Pbarzed^N /2^{N+1} \Pbarzed
\]
with cokernel which surjects to $p_N \Ibar$, by Lemma
\ref{lem:Pbarzed_filt_quot_surject};  this gives the inequality:
\begin{eqnarray}
\label{eqn:inequality}
 \big[
\Pbarzed^N /2^{N+1} \Pbarzed
\big]
\geq 
\big[
\Pbarzed^{N+1} / 2^{N+1} \Pbarzed
\big]
+ 
\big[
p_N\Ibar 
\big],
\end{eqnarray}
with equality if and only if the second statement
holds.

The short exact sequence
\[
 0 
\rightarrow 
\Pbar \cong (2^N \Pbarzed
/2^{N+1} \Pbarzed )
\rightarrow 
\Pbarzed^N
/2^{N+1} \Pbarzed 
\rightarrow 
\Pbarzed^N
/2^{N} \Pbarzed 
\rightarrow 
0
\]
and the inductive hypothesis give:
\begin{eqnarray*}
 \big[ 
\Pbarzed^{N} / 2^{N+1} \Pbarzed
\big] 
&=& 
[\Pbar ] + 
\sum _{i=1}^N
\big[
\Pbar^{i}
\big].
\end{eqnarray*}
Now 
$\big[
\Pbar 
\big ] 
= 
\big[ \Pbar ^{N+1} \big ] 
+   
\big[p_N \Ibar\big]
$, hence (\ref{eqn:inequality}) implies that
\[
\sum _{i=1}^{N+1}
\big[
\Pbar^{i}
\big] \geq \big[\Pbarzed^{N+1} / 2^{N+1} \Pbarzed \big],
\]
with equality if and only if the second statement
holds. Hence, both inequalities are equalities and  the inductive step is established.
\end{proof}

\begin{cor}
\label{cor:completion}
For $V \in \obj \fdvs$,  the topologies on the abelian
group $\Pbarzed(V)$ induced by the $2$-adic filtration
$2^i \Pbarzed$ and by the filtration $\Pbarzed^i$ are equivalent. 
\end{cor}

\begin{proof}
By Lemma \ref{lem:2_times_filt}, $2^n \Pbarzed \subset \Pbarzed^{n+1}$; conversely, it is straightforward to show, using Proposition \ref{prop:subquotients_Pbarzed}, that
 $
 \Pbarzed^{n+k} (\field^n) \subset 2^k \Pbarzed(\field^n).
$
\end{proof}

\subsection{The structure of the  quotients $\Pbarzed / \Pbarzed^{n+1}$}
\label{subsect:R_functors}

\begin{nota}
 For $n  \in \nat$, let $R^n_\zed$ denote the quotient $\Pbarzed / \Pbarzed^{n+1}$.
\end{nota}

\begin{lem}
\label{lem:Q_basis_properties}
For $n>0$ an integer: 
\begin{enumerate}
\item
$2^{n+1}R^n_\zed=0$ and $R^n _\zed (\field) \cong \zed/2^n$;
\item 
there are short exact sequences:
\begin{eqnarray*}
 &&
 0 
\rightarrow 
p_n \Ibar \rightarrow
R^n _\zed \rightarrow
R^{n-1}_\zed
\rightarrow 0
\\
&&
 0
\rightarrow
R^{n-1}_\zed
\rightarrow
R^n 
_\zed
\rightarrow
q_n\Pbar
\rightarrow
0;
\end{eqnarray*}
\item
the largest subfunctor $_2R^n _\zed $ of $R^n _\zed$  annihilated by $2$ is isomorphic to $p_n\Ibar$.
\end{enumerate}
\end{lem}

\begin{proof}
The first statement is clear and the first short exact sequence is provided by Proposition
\ref{prop:subquotients_Pbarzed}. The second short exact sequence is induced by the inclusion $2 \Pbarzed
\hookrightarrow \Pbarzed$, since Proposition  \ref{prop:subquotients_Pbarzed} implies that $2 \Pbarzed \cap
\Pbarzed^{n+1}$ is isomorphic to $\Pbarzed^n$ under the isomorphism $\Pbarzed \cong 2 \Pbarzed$.

The proof that the largest subfunctor of $R^n _\zed$ annihilated by $2$ is $p_n \Ibar$ is by induction on $n$; for $n=1$, $R^1_\zed \cong \Lambda^1$ and the result is immediate.
For the inductive step, the short exact sequence
\[ 0
\rightarrow
p_n \Ibar \rightarrow
R^n _\zed \rightarrow
R^{n-1}_\zed
\rightarrow 0,
\]
implies that there is an exact sequence
\[
	0
\rightarrow
p_n \Ibar \rightarrow
\ _2R^n _\zed \rightarrow
p_{n-1} \Ibar,
\]
where the right hand term is given by the inductive hypothesis. To complete the result, it suffices to show that the image of $_2R^n _\zed$ in $p_{n-1} \Ibar$ is trivial; hence, by Lemma \ref{lem:subfunctor_pdI_qdP}, it suffices to show this after evaluation on $\field$. This follows  from the fact that $R^n _\zed (\field) \cong \zed/2^n$.
\end{proof}

\begin{lem}
\label{lem:Q_simple_head_socle}
For $n>0$ an integer, the functor $R^n _\zed$ has simple
head $\Lambda^1$ and simple socle $\Lambda^1$. 
\end{lem}

\begin{proof}
 The functor $R^n_\zed$ is a non-trivial quotient of $\Pbarzedtwo$, which has simple head $\Lambda^1$ by
Lemma \ref{lem:proj_mod_2}, hence $R^n _\zed$ has simple head $\Lambda^1$. 

The proof that the socle is $\Lambda^1$ is by induction on $n$, starting from
the case $n=1$, which is clear, since $R^1_\zed \cong \Lambda ^1$ is simple. 
The exact sequence $0 \rightarrow p_n\Ibar \rightarrow R^n _\zed \rightarrow
R^{n-1} _\zed \rightarrow 0$ shows that the socle of $R^n_\zed$ is either
$\Lambda^1$ or $\Lambda^1 \oplus \Lambda^1$. The latter possibility is excluded
since $R^n_\zed (\field) = \zed/2^n$, by Lemma \ref{lem:Q_basis_properties}.
\end{proof}

\begin{thm}
\label{thm:Q-self-dual}
 Let $n>0$ be an integer. The functor $R^n _\zed$ is self-dual;
more precisely, any surjection 
 $
 \Pbarzedtwo \twoheadrightarrow 
DR^n _\zed
$ 
factors canonically across an isomorphism $R^n _\zed
\stackrel{\cong}{\rightarrow} DR^n _\zed$.
\end{thm}

\begin{proof}
 The proof is by induction upon $n$, starting from the case $n=1$, when $R^n
_\zed$ is the simple functor $\Lambda^1$, which is self-dual. The factorization
statement follows from the fact that $\Pbarzedtwo$ has simple head.

For the inductive step, $R^n _\zed$ has simple socle $\Lambda ^1$, hence $DR^n
_\zed$ has simple head $\Lambda^1$ and there exists a surjection $\Pbarzedtwo
\twoheadrightarrow DR^n _\zed$ (by projectivity of $\Pbarzed$). This gives rise to a morphism of short exact
sequences
\[
 \xymatrix{
0
\ar[r]
&
\Pbarzedtwo
\ar[r]^2
\ar@{->>}[d]
&
\Pbarzedtwo
\ar@{->>}[d]
\ar[r]
&
\Pbar
\ar@{->>}[d]
\ar[r]
&
0
\\
0
\ar[r]
&
DR^{n-1}_\zed
\ar[r]
&
D R^n _\zed
\ar[r]
&
q_n \Pbar
\ar[r]
&
0,
}
\]
 where the lower exact sequence is the dual of the first exact sequence of
Lemma \ref{lem:Q_basis_properties}, the commutativity of the right hand square
follows from the fact that there is a unique  non-trivial 
morphism $\Pbarzedtwo \rightarrow  q_n\Pbar$ and the surjectivity of the left hand
vertical morphism is seen by evaluating on $\field$, since $DR^{n-1}_\zed$ has simple head.

By the inductive hypothesis, the left hand vertical morphism factorizes across
an isomorphism $R^{n-1}_\zed \stackrel{\cong}{\rightarrow} DR^{n-1}_\zed$. In
particular, this implies that $\Pbarzedtwo^{n+1}$ lies in the kernel of
$\Pbarzedtwo\twoheadrightarrow  DR^n _\zed$, hence this induces a surjection 
\[
 R^n_\zed \twoheadrightarrow DR^n_\zed
\]
which is an isomorphism, since the objects have finite composition
series with isomorphic associated graded functors.
\end{proof}

\begin{defn}
Let $R^\infty_\zed$ be the direct limit of the diagram
\[
 R^1 _\zed \hookrightarrow R^2_\zed \hookrightarrow R^3 _\zed \hookrightarrow
\ldots
\]
of monomorphisms provided by Lemma \ref{lem:Q_basis_properties}. 
\end{defn}

\begin{thm}
\label{thm:pontrjagin_duality}
 There are Pontrjagin duality isomorphisms:
\begin{eqnarray*}
 R^\infty_\zed &\cong& D \Pbarzedtwo \\
\Pbarzedtwo  & \cong & D R^\infty_\zed.
\end{eqnarray*}
\end{thm}

\begin{proof}
Observe that the functor $R^\infty_\zed$ takes values in torsion $2$-groups.  By construction and Corollary \ref{cor:completion}, the functor
$\Pbarzedtwo$ is isomorphic to the inverse limit of the natural system
of quotients 
\[
 \ldots \twoheadrightarrow  R^n_\zed \twoheadrightarrow R^{n-1}_\zed
\twoheadrightarrow \ldots \twoheadrightarrow R^1_\zed \cong \Lambda^1.
\]
Applying the Pontrjagin duality functor and  using the fact that each $R^n
_\zed$ is self dual, gives a direct system which is isomorphic to that defining
$R^\infty_\zed$ (the latter fact follows from the proof of Theorem \ref{thm:Q-self-dual}). The result follows
from  Pontrjagin duality for abelian groups.
\end{proof}

\section{Milnor derivations}
\label{sect:milnor}

This section establishes the fundamental ingredient, Proposition \ref{prop:homology_K_Q1}, to the proof of an algebraic version of Ossa's theorem; the prime $p$ is taken to be $2$.

\subsection{Milnor derivations on symmetric powers}

\begin{defn}
 For $i\in \nat$, let $Q_i : S^1 \rightarrow S^{2^{i+1}}$ denote the
iterated Frobenius $x \mapsto x^{2^{i+1}}$
and also its extension $S^ n  \stackrel{Q_i}{ \rightarrow }S^{n+2^{i+1}
-1}$ to a derivation of $\bigoplus
S^*$,  defined by the composite
\[
 S^n 
\stackrel{\Delta}{\rightarrow }  S^{n-1} \otimes S^1
\stackrel{1 \otimes Q_i } {\rightarrow}
S^{n-1} \otimes S^{2^ {i+1}}
\stackrel{\mu}{\rightarrow }
S^{n +2^{i+1}
-1}.
\]
\end{defn}

\begin{lem}
\label{lem:Q_i_commute}
For $i,j \in \nat$, $Q_i \circ Q_i =0$ and the derivations $Q_i, Q_j$ commute. Hence the graded algebra $\bigoplus S^* $ is a
module in $\f$ over the exterior algebra $\Lambda (Q_i | i \geq 0)$;  in particular,  
 $(S^*, Q_i)$ has the structure of a commutative differential graded algebra in $\f$. 
\end{lem}

\begin{proof}
Straightforward.
\end{proof}

\begin{nota}
 For $j>0$ an integer, let $S^* /\langle x^{2^j}\rangle$ denote the
truncated symmetric power functor, so that 
$S^n /\langle x^{2^j}\rangle$ is the cokernel of the composite 
\[
 S^{n-2^j} \otimes S^1 \stackrel{1 \otimes Q_{j-1}}{\longrightarrow}
S^{n-2^j} \otimes S^{2^j} 
\stackrel{\mu}{\rightarrow}
S^n.
\]
\end{nota}

In the following statement, the degree corresponds to the  grading inherited from that of $S^*$.

\begin{prop}
\label{prop:homology_Qi}
For $i\in \nat$, the homology $H(S^*, Q_i)$ is the
truncated symmetric power functor $ S^*/
\langle x^{2^i} \rangle$, concentrated in even degrees. Explicitly:
\[
 H (S^* , Q_i) ^k \cong 
\left\{
\begin{array}{ll}
 0 & k \equiv 1 \mod 2\\
S^d /\langle x^{2^i} \rangle  & k = 2d .
\end{array}
\right.
\]
\end{prop}

\begin{proof}
 The result follows from the calculation of the homology of the differential
graded algebra $(\field [x], d x =
x^{2^{i+1}})$, which is the truncated polynomial algebra $\field [y] /
y^{2^{i}}$, where $y = x^2$. The homology of 
the tensor product of such algebras is  calculated by using the Künneth theorem.
It remains to show that this corresponds to the stated  functorial
isomorphism.

The above  establishes that the homology is concentrated in even degrees;
moreover, in degree $k= 2d$, the
Frobenius $S^d \hookrightarrow S^{2d}$ maps to the cycles in degree $k$ and induces a surjection onto the homology.

It is straightforward to check that, for integers $i, d \geq 1$, there is a commutative diagram  in $\f$:
\[
 \xymatrix{
S^{d - 2^{i}}\otimes S^1 
\ar[r]^{\Phi \otimes 1}
\ar[d]_{1 \otimes Q_{i-1}}
&
S^{2d - 2^{i+1}} \otimes S^1 
\ar[r]^\mu
&
S^{2d - 2^{i+1} +1 }
\ar[d]^{Q_i}
\\
S^{d - 2^{i}}\otimes S^{2^i}
\ar[r]_\mu
&
S^d
\ar[r]_\Phi
&
S^{2d},
}
\]
where $\Phi$ is the Frobenius. It follows that  the morphism $S^d \rightarrow H (S^* , Q_i)^{2d}$ factorizes across the
canonical projection $S^ d \twoheadrightarrow S^d/ \langle x^{2^i} \rangle$.
This completes the proof.
\end{proof}

\begin{rem}
For  $i=0$, the homology is $\field$ concentrated in degree
zero; in particular, the $Q_ 0 :S^n \rightarrow S^{n+1}$ induce an exact complex
\[
 0 \rightarrow S^1 \rightarrow S^2 \rightarrow S^3 \rightarrow \ldots \ .
\]
For $i=1$, there is non-trivial homology in even degrees; the homology of the
complex 
 $
 S^{2d -3} \rightarrow S^{2d} \rightarrow S^{2d +3} 
$ is  $\Lambda^d $. Moreover, the proof of the Proposition gives an exact  complex
\[
 S^d \oplus S^{2d-3} 
\stackrel{(\Phi , Q_1) } {\rightarrow} 
S^{2d}
\stackrel{Q_1}{\rightarrow}
S^{2d +3}.
\]
\end{rem}

\subsection{The $Q_0$-kernel complex}
\label{subsect:Q0_kernel}

\begin{nota}
\label{nota:Kn}
 For $n\in \nat$, let $K_n$ denote the kernel of $Q_0 : S^n \rightarrow
S^{n+1}$.
\end{nota}

By construction, there are short exact sequences $0 \rightarrow K_n \rightarrow
S^n \rightarrow K_{n+1} \rightarrow 0$,
for $n\geq 0$.
Initial values of $K_n$ are  $K_0 = \field = S^0$, $K_1 = 0$,
$K_2 = \Lambda^1= S^1 \hookrightarrow
S^2$, $K_3 = \Lambda  ^2$.

\begin{lem}
\label{lem:exact_complexes}
 The derivation $Q_1$ on $S^*$  induces a differential $Q_1 : K_n \rightarrow
K_{n+3}$ and there is a short exact
sequence  of complexes 
\[
 0 
\rightarrow 
(K_{i+3\bullet}, Q_1) 
\rightarrow 
(S^{i+3\bullet}, Q_1) 
\rightarrow 
(K_{i+1 + 3\bullet}, Q_1) 
\rightarrow 
0,
\]
where $0 \leq i <3$ and $\bullet \geq 0$. 
\end{lem}

\begin{proof}
 A consequence of the exactness of the $Q_0$-complex in positive dimensions and
the  commutation of  $Q_0, Q_1$ (Lemma \ref{lem:Q_i_commute}).
\end{proof}

\begin{nota}
For $n\in \nat$, let $L_n$ denote the image of $Q_1 : K_{n-3}
\rightarrow K_n$ and $\tilde{L}_n $ the kernel of $Q_1 : K_n
\rightarrow K_{n+3}$.
\end{nota}

The following is clear:

\begin{prop}
 \label{prop:Ltilde_mult}
The graded functors $K_*$,  $\tilde{L}_*$ have  unique  graded algebra structures such
that the canonical inclusions 
$\tilde{L}^* \hookrightarrow K_* \hookrightarrow S^*$ are  morphisms of  commutative graded
algebras in $\f$.
\end{prop}

\begin{prop}
\label{prop:homology_K_Q1}
 The homology of the complexes $(K_{i+3\bullet} , Q_1)$ is determined by 
\[
\tilde{L}_n /L_n
 \cong  
\left\{
\begin{array}{ll}
\field & n=0
\\
0& n \equiv 1 \mod 2\\
 p_d\Ibar & n = 2d >0. 
\end{array}
\right.
\]
\end{prop}

\begin{proof}
The proof is by induction upon $n$; the case $n=0$ is clear. It is straightforward to show that the odd degree homology is trivial (independently of the calculation of the even degree homology).

For the inductive step, consider the commutative diagram
arising from the short exact sequence of complexes given by Lemma \ref{lem:exact_complexes}:
\[
\xymatrix{
& 
S^{n-6}
\ar[r]
\ar[d]
& 
K_{n-5} 
\ar[d]
\\
K_{n-3} \ar[r]
\ar[d] 
&
S^{n-3} 
\ar[r]
\ar[d]
&
K_{n-2}
\ar[d]
\\
K_n \ar[r]
\ar[d]
&
S^n 
\ar[r]
\ar[d]
&
K_{n+1}
\ar[d]
\\
K_{n+3} 
\ar[r]
&
S^{n+3}
\ar[r]
&
K_{n+4};
}
\]
the homology $H$ at the middle of  the left hand column is calculated  in terms of the known
homologies in the
other two columns, via the long exact sequence in homology.

For $n=2d >0$, there is a short exact sequence 
\[
 0 
\rightarrow 
p_{d-1}\Ibar 
\rightarrow H
\rightarrow 
\Lambda^d 
\rightarrow 
0,
\]
 by the inductive hypothesis for the left hand term and Proposition
\ref{prop:homology_Qi} for the exterior power, using the vanishing of homology
in odd degrees.  It suffices to show that $ p_d \Ibar$ is a subquotient of  $H$.

The Frobenius $\Phi : S^d \hookrightarrow S^{2d}$ maps to $K_{2d}$ and the image
lies in the kernel $\tilde{L}_{2d}$ of $K_{2d}
\rightarrow K_{2d+3}$. There is a unique non-trivial morphism $S^d
\rightarrow \Ibar$ and this  has
image $p_d\Ibar$; since $\Ibar$ is injective, this extends to
give a commutative diagram 
\[
 \xymatrix{
&
K_{2d - 3} 
\ar[d]
\\
S^d 
\ar@{^(->}[r]
\ar@{->>}[d]
&
\tilde{L}_{2d}
\ar[d]
\\
p_d \Ibar
\ar@{^(->}[r]
&
\Ibar 
.
}
\]
The composite morphism $K_{2d - 3} \rightarrow \Ibar$ is trivial, since
$K_{2d-3} (\field) =0$; it follows that
$p_d\Ibar $ is a subquotient of the homology $H$, as required. 
\end{proof}

\begin{exam}
\label{exam:L_low_degree}
 It is straightforward to verify that $L_n=0$ for $n \leq 5$, hence the initial values of $\tilde{L}_i$ are  given by 
\[
 \tilde{L}_n = \left\{
\begin{array}{ll}
 0 & n \in \{ 1, 3, 5 \}\\
\Lambda^1 & n=2 \\
S^2 & n =4.
\end{array}
\right.
\]

\end{exam}

\section{Connective complex K-cohomology and homology}
\label{sect:ku}

A complete functorial description of both $V \mapsto ku^* (BV_+)$ and $V \mapsto ku_* (BV_+)$ is given, using the results of Section \ref{sect:groupring} and Section \ref{sect:milnor}.

\subsection{Recollections}

The Postnikov towers of $ku$ and $KU$ provide morphisms  
$
ku \rightarrow
H\zed 
$ 
 and $ku \rightarrow KU$  of commutative  ring spectra relating connective (resp. periodic) complex $K$-theory $ku$ (resp. $KU$) and the integral Eilenberg-MacLane spectrum $H\zed$.

There are  cofibre sequences 
$
\Sigma^2 ku 
\stackrel{v}{\rightarrow}
ku 
\rightarrow 
H \zed
$, $
H \zed 
\stackrel{2}{\rightarrow}
H \zed 
\stackrel{\rho}{\rightarrow}
H \field,
$
where $v$ is multiplication by  the Bott element ($ku_* \cong \zed [v]$, with $|v|=2$).
These give rise to (generalized) Bocksteins; in
particular:

\begin{nota}
 Let $\qfrak$ denote the first $k$-invariant of $ku$, given by the composite $H \zed  \rightarrow \Sigma^3
ku\rightarrow \Sigma^3 H \zed$.
\end{nota}

Recall that the Milnor derivation $Q_1 \in H\field^3 H \field$ is the commutator $[Sq^2 , Sq^1]$ (the Milnor derivation $Q_0$ is the Bockstein  $\beta = Sq^1$).

\begin{lem}
 \label{lem:Q1_qfrak}
There is a commutative diagram:
\[
 \xymatrix{
H\zed \ar[r]^{\qfrak} 
\ar[d]_\rho
&
\Sigma^3 H\zed 
\ar[d]^\rho
\\
H \field \ar[r] _{Q_1}
&
\Sigma ^3 H\field. 
}
\]
\end{lem}

\begin{proof}
 The morphism $\qfrak$ is  the image under the integral Bockstein $H\field^2 H\zed \rightarrow H\zed^3 H\zed$ of the class of $Sq^2$ (recall that $H\field^* H \zed \cong \cala/ (Sq^1) $)  \cite[proof of III.16.6]{adams}. Hence there is a commutative diagram 
\[
 \xymatrix{
H \zed 
\ar@/^1pc/[rr]^{\qfrak}
\ar[d]_\rho
\ar[r]
&
\Sigma^3 ku 
\ar[r]
&
\Sigma^3 H \zed 
\ar[d]^\rho
\\
H\field 
\ar[r]^{Sq^2} 
\ar@/_1pc/[rr]_{Sq^3}
&
\Sigma^2 H\field
\ar[r]^{Sq^1}
\ar[ur]
&
\Sigma^3 H \field.
}
\]
Since the composite $Sq^1 \circ \rho$ is trivial, the result follows.
\end{proof}

\begin{nota}
 The morphisms in (co)homology induced by $\qfrak$ (respectively $Q_1$) will be denoted simply $\qfrak$ (resp. $Q_1$).
\end{nota}

\begin{lem}
\label{lem:Bockstein_exact_Q1_qfrak}
 For $Y$ a spectrum, the following conditions are equivalent:
\begin{enumerate}
 \item 
$H\zed^* Y \stackrel{\rho} {\rightarrow} H \field^* Y $ is a monomorphism;
\item
the Bockstein complex $(H\field^* Y , \beta) $ is exact;
\item
 $H\zed_* Y \stackrel{\rho} {\rightarrow} H \field_* Y $ is a monomorphism;
\item
the Bockstein complex $(H\field_* Y , \beta) $ is exact.
\end{enumerate}
When these conditions are satisfied, the respective morphisms $\qfrak$  are determined by the commutative diagrams:
\[
 \xymatrix{
H\zed^* Y 
\ar[r]^{\qfrak} 
\ar@{^(->}[d]_\rho
&
H\zed^{*+3} Y 
\ar@{^(->}[d]^{\rho}
&
H\zed_* Y 
\ar[r]^{\qfrak} 
\ar@{^(->}[d]_\rho
&
H\zed_{*-3} Y 
\ar@{^(->}[d]^{\rho}
\\
H\field^* Y 
\ar[r]_{Q_1} 
&
H\field^{*+3} Y 
&
H\field_* Y
\ar[r]_{Q_1} 
&
H\field_{*-3} Y.
}
\]
\end{lem}

\begin{proof}
 The equivalence of the conditions is standard and the commutative diagrams follow from Lemma \ref{lem:Q1_qfrak}; these determine $\qfrak$, since the vertical morphisms are injective, by hypothesis. 
\end{proof}

\begin{exam}
 \label{exam:BV_2torsion}
The hypotheses of Lemma \ref{lem:Bockstein_exact_Q1_qfrak} are satisfied for $Y = \Sigma^\infty BV$, where $V$ is  an elementary abelian $2$-group. In particular, the action of $\qfrak$ on $H \zed^* (BV)$ can be understood in terms of the action of $\Lambda (Q_0, Q_1)$ on $H\field^* (BV)$. 
\end{exam}

\subsection{On $v$-torsion and cotorsion}

There is a hereditary torsion theory $(\mathcal{T}, \mathcal{F})$ on the category of $\zed [v]$-modules where $\mathcal{T}$ is the category of $v$-torsion modules (every element is annihilated by some power of $v$) and $\mathcal{F}$ the category of $v$-cotorsion modules (those $\zed [v]$-modules $M$ which embed in $M[\frac{1}{v}]$). This torsion theory has associated torsion functor $\torsion$ and cotorsion functor $\cotorsion$ so that, for a $\zed [v]$-module, there is a natural short exact sequence of $\zed [v]$-modules:
\[
 0 
\rightarrow 
\torsion M 
\rightarrow 
M 
\rightarrow 
\cotorsion M
\rightarrow 
0. 
\]
The submodule of $M$ annihilated by $v$ is written $\annih M$.

\begin{lem}
\label{lem:ses_torsion_cotorsion}
For a $\zed [v]$-module $M$, there is an isomorphism $\torsion M \cap v M \cong v \torsion M$, hence $\torsion M \hookrightarrow M$ induces a monomorphism $(\torsion M)/v \hookrightarrow M/v$. 

Moreover, there is a short exact sequence of $\zed [v]$-modules:
\[
 0 
\rightarrow 
\cotorsion M 
\stackrel{v} {\rightarrow} 
\cotorsion M
\rightarrow 
(M/v)/ \big((\torsion M)/v\big)
\rightarrow 0.
\]
\end{lem}

\begin{proof}
 Straightforward.
\end{proof}

For  a $\zed [v]$-module $M$, there is a composite 
$$\annih M \hookrightarrow \torsion M \twoheadrightarrow (\torsion M)/v.$$

The following algebraic result is required in the proof of Proposition \ref{prop:qfrak_formal}.

\begin{lem}
 \label{lem:ann_torsion}
For $M$  a graded $\zed[v]$-module, the following conditions are equivalent:
\begin{enumerate}
 \item 
$\annih M = \torsion M$;
\item
$\annih M \rightarrow (\torsion M)/v$ is injective.
\end{enumerate}
If, for each degree $d$, there exists $N(d) \in \nat$ such that 
\[
 (v^{N(d)} M \cap \torsion M) ^d =0
\]
  then these conditions are equivalent to 
\begin{enumerate}
 \item[(3)]
$\annih M \rightarrow (\torsion M)/v$ is surjective.
\end{enumerate}
\end{lem}

\begin{proof}
 If $\annih M = \torsion M$, then $\torsion M \cong (\torsion M)/v$ and hence condition (1) implies both (2) and (3).

Suppose condition (2) holds and consider $x \in \torsion M$, so that there exists $t \in \nat$ such that $v^t x \neq 0$ and $v^{t+1} x=0$. Hence $v^t x \in \annih M$; the hypothesis implies that $v^t x$ is not in $v M$, hence $t=0$ and $x \in \annih M$, as required.

Now suppose that condition (3) holds, under the additional hypothesis. Consider $x \in \torsion M$; by surjectivity of $\annih M \twoheadrightarrow (\torsion M)/v$, $x = a + vy$, for $a \in \annih M$ and $y \in \torsion M$. An induction shows that, for  $0< n \in \nat$,  $x = a + v^n y_n$, for some $y_n \in \torsion M$; it suffices to show that $v^n y_n = 0$ for $n \gg 0$. By construction $v^n y_n \in (v^n M \cap \torsion M)^{|x|}$, hence the hypothesis implies that the element is zero for $n \geq N(|x|)$.
\end{proof}

\begin{rem}
 The example $\zed [v]/v^\infty$ shows that the conditions (1) and (3) are  not equivalent without the additional hypothesis.
\end{rem}

The following Lemma applies when considering $ku$-cohomology of a spectrum and allows the cotorsion to be related to periodic $K$-theory.

\begin{lem}
\label{lem:alg_localization_periodic}
\cite[Chapter 1]{bruner_greenlees}
 If $Z$ is a connective spectrum, there is a natural isomorphism $KU^* (Z) \cong ku^* (Z) [\frac{1}{v}]$.
\end{lem}

For the applications, $Z$ will be the suspension spectrum of a space, hence the connectivity hypothesis is not restrictive. 

\begin{lem}
\label{lem:ses_tors_cotors}
 For $X$  a spectrum:
\begin{enumerate}
 \item 
there is a natural exact sequence of $\zed[v]$-modules:
\[
 0 
\rightarrow 
\torsion ku_*(X)
\rightarrow 
ku_* (X)
\rightarrow 
\cotorsion ku_* (X)
\rightarrow 
0
\]
and a natural isomorphism $$\cotorsion ku_* (X) \cong \mathrm{image} \{ ku_* (X) \rightarrow KU_* (X)\};$$ 
\item
there is a natural exact sequence of $\zed[v]$-modules:
\[
 0 
\rightarrow 
\torsion ku^*(X)
\rightarrow 
ku^* (X)
\rightarrow 
\cotorsion ku^* (X)
\rightarrow 
0
\]
and, if $X$ is connective, a natural isomorphism $$\cotorsion ku^* (X) \cong \mathrm{image} \{ ku^* (X) \rightarrow KU^* (X) \}.$$
\end{enumerate}
\end{lem}

\begin{lem}
\label{lem:ses_ku_HZ}
 For $X$ a spectrum, there are natural short exact sequences 
\begin{eqnarray*}
 0 \rightarrow ku_* (X) / v \rightarrow H\zed_* (X) \rightarrow \annih ku_{*-3} (X) 
\rightarrow 0;
\\
0 \rightarrow ku^* (X) / v \rightarrow H\zed^* (X) \rightarrow \annih ku^{*+3} (X) 
\rightarrow 0.
\end{eqnarray*}
Moreover, there are natural inclusions
\begin{eqnarray*}
 \mathrm{Im} \qfrak
\subset 
\big(\torsion ku_* (X)\big) / v 
\subset 
ku_* (X) / v
\subset 
\mathrm{Ker} \qfrak;
\\
\mathrm{Im} \qfrak 
\subset 
\big(\torsion ku^* (X)\big) / v 
\subset 
ku^* (X) / v
\subset 
\mathrm{Ker} \qfrak.
\end{eqnarray*}
\end{lem}

\begin{proof}
 The proofs for homology and cohomology are formally the same, hence consider $ku$-homology. The short exact sequence is induced by the cofibre sequence $\Sigma^2 ku \rightarrow ku \rightarrow H \zed$. The inclusion $(\torsion ku_* (X)) / v 
\subset  ku_* (X) / v$ is provided by Lemma \ref{lem:ses_torsion_cotorsion};  the outer inclusions are clear, from the definition of $\qfrak$.
\end{proof}

\begin{prop}
\label{prop:qfrak_formal}
Let $Z$ be a connective spectrum. 
\begin{enumerate}
 \item 
The following conditions are equivalent:
\begin{enumerate}
\item
$\annih ku_* (Z) = \torsion ku_* (Z)$;
\item
$\big(\torsion ku_* (Z)\big)/v \cong \mathrm{Im} \qfrak$;
\item
$ku_* (Z)/v \cong \mathrm{Ker}\qfrak $.
\end{enumerate}
If these conditions are satisfied, then 
$(ku_* (Z)/v) / \big((\torsion ku_* (Z))/v\big) \cong \mathrm{Ker}\qfrak / \mathrm{Im}\qfrak$ and there is a short exact sequence of $\zed[v]$-modules:
\[
 0 
\rightarrow 
\cotorsion ku_{*-2} (Z) 
\stackrel{v}{\rightarrow}
\cotorsion ku_{*} (Z) 
\rightarrow 
 \mathrm{Ker}\qfrak / \mathrm{Im}\qfrak
\rightarrow 0
\]
and  a pullback:
\[
 \xymatrix{
0 
\ar[r]
&
\torsion ku_* (Z)
\ar[r]
\ar[d]_\cong
&
ku_{*}(Z) 
\ar[r]
\ar[d]
&
\cotorsion ku_* (Z) 
\ar@{->>}[d]
\ar[r]
&
0
\\
0 \ar[r]
&
\mathrm{Im}\qfrak
\ar[r]
&
\mathrm{Ker}\qfrak  
\ar[r]
&
\mathrm{Ker}\qfrak  /\mathrm{Im}\qfrak
\ar[r]
&
0.
}
\]
\item
Suppose that, for each degree $d$, there exists an integer $N(d)$ such that $(  v^{N(d)} ku^* (Z)\cap\torsion ku^* (Z) ) ^d =0$, then the following conditions are equivalent:
\begin{enumerate}
\item
$\annih ku^* (Z) = \torsion ku^* (Z)$;
\item
$\big(\torsion ku^* (Z)\big)/v \cong \mathrm{Im} \qfrak$;
\item
$ku^* (Z)/v \cong \mathrm{Ker}\qfrak $.
\end{enumerate}
If these conditions are satisfied, then 
$(ku^* (Z)/v) / \big((\torsion ku^* (Z))/v\big) \cong \mathrm{Ker}\qfrak / \mathrm{Im}\qfrak$ and  there is a short exact sequence of $\zed[v]$-modules:
\[
 0 
\rightarrow 
\cotorsion ku^{*+2} (Z) 
\stackrel{v}{\rightarrow}
\cotorsion ku^{*} (Z) 
\rightarrow 
 \mathrm{Ker}\qfrak / \mathrm{Im}\qfrak
\rightarrow 0
\]
and  a pullback:
\[
 \xymatrix{
0 
\ar[r]
&
\torsion ku^* (Z)
\ar[r]
\ar[d]_\cong
&
ku^{*}(Z) 
\ar[r]
\ar[d]
&
\cotorsion ku^* (Z) 
\ar@{->>}[d]
\ar[r]
&
0
\\
0 \ar[r]
&
\mathrm{Im}\qfrak
\ar[r]
&
\mathrm{Ker}\qfrak  
\ar[r]
&
\mathrm{Ker}\qfrak  /\mathrm{Im}\qfrak
\ar[r]
&
0.
}
\]
\end{enumerate}
\end{prop}

\begin{proof}
The hypotheses imply that, for $M \in \{ ku_* (Z) , ku^* (Z) \}$, the three conditions of Lemma \ref{lem:ann_torsion} are equivalent.

The equivalence of the conditions (a), (b), (c) follows from an analysis of the short exact sequences of Lemma \ref{lem:ses_ku_HZ};  consider $ku$-homology (the argument for $ku$-cohomology is similar), so that there is  a commutative diagram:
\[
 \xymatrix{
&
&
\mathrm{Ker} \qfrak
\ar@{^(->}[d]
\\
0 
\ar[r]
&
ku_* (Z) / v 
\ar[r]
\ar@{^(->}[ru]
&
H\zed_* (Z) 
\ar@{->>}[d]
\ar[r]
&
\annih ku_{*-3} (Z) 
\ar[r]
\ar@{->>}[ld] 
&
0
\\
& 
(\torsion ku_* (Z) ) / v
\ar@{^(->}[u]
&
\mathrm{Im} \qfrak, 
\ar@{^(->}[l] 
}
\]
in which the middle row and column are both short exact.

By the five lemma,  $\mathrm{Im} \qfrak  \cong \annih ku_{*-3} (Z) $ if and only if $\mathrm{Ker} \qfrak \cong ku_* (Z) / v $. Moreover, Lemma \ref{lem:ann_torsion} implies that the following three conditions are equivalent:
\begin{enumerate}
 \item 
$\annih ku_* (Z) = \torsion ku_* (Z)$;
\item 
$\mathrm{Im} \qfrak  \cong (\torsion ku_* (Z))/v $;
\item
$\mathrm{Im} \qfrak  \cong \annih ku_{*-3} (Z) $. 
\end{enumerate}
This shows that conditions (a), (b), (c) are equivalent. The consequences follow, using Lemma \ref{lem:ses_torsion_cotorsion} to provide the short exact sequence which calculates $\cotorsion ku_* (Z) /v$.
\end{proof}

\begin{rem}\ 
\begin{enumerate}
 \item 
Under the hypotheses of the Proposition,  $ku_* (Z)$ (respectively $ku^* (Z)$) is determined by $(H\zed_ * (Z) , \qfrak)$ (resp. $(H \zed^* (Z),\qfrak)$), up to the analysis of the $v$-adic filtration of $\cotorsion ku_* (Z)$ (resp. $\cotorsion ku^*(Z)$).
\item
If  $H\zed_* (Z) \rightarrow H \field _* (Z)$ is a monomorphism (so that the hypotheses of Lemma \ref{lem:Bockstein_exact_Q1_qfrak} are satisfied), then the data is provided by $H \field_* (Z)$, considered as a $\Lambda (Q_0, Q_1)$-module and there is a form of duality between $ku_* (Z)$ and $ku^* (Z)$.
\item
Related considerations for connected Morava $K$-theories occur in  \cite{lellmann}.
\end{enumerate}
\end{rem}

The nilpotency hypothesis of Proposition \ref{prop:qfrak_formal} is supplied by the following result when considering the $ku$-cohomology of spaces.

\begin{lem}
\label{lem:torsion_bound}\cite[Lemma 1.5.8]{bruner_greenlees}
 Let $Y$ be a space such that $ku^* (Y_+)$ is a Noetherian $\zed[v]$-algebra. Then there exists a natural
number $N$ such that $$v^N \torsion ku^* (Y_+) =0.$$
\end{lem}

\begin{exam}
 For $G$ a finite group, $ku^* (BG_+)$ is a Noetherian $\zed [v]$-algebra (see \cite[Section 1.1]{bruner_greenlees} for example).
\end{exam}

When considering the unreduced $ku$-cohomology of a space, under the cohomological hypothesis of Proposition \ref{prop:qfrak_formal}, the following result is clear.

\begin{prop}
\label{prop:mult_ku_cohom}
(Cf. \cite{bruner_greenlees}.)
 Let $Y$ be a space such that $\torsion ku^* (Y_+) \cong \annih ku^* (Y_+)$, then the algebra structure of $ku^* (Y_+)$ is determined by the induced algebra structures of $ku^* (Y_+)/v$ and $\cotorsion ku^* (Y_+)$. In particular, there is a monomorphism of algebras 
\[
 ku^* (Y_+) 
\rightarrow 
KU^* (Y_+)
\prod 
H\zed^* (Y_+)
\]
where 
$\cotorsion ku^* (Y_+)$ is considered as a subalgebra of $KU^* (Y_+)$ and  $ku^* (Y_+)/ v  \cong \mathrm{Ker} \qfrak$, as a subalgebra of $H\zed^* (Y_+)$. 
\end{prop}

\subsection{The $ku$-cohomology of $BV_+$}
\label{subsect:ku_cohom_BV}

The above discussion applies in considering the group $ku$-cohomology $ku^* (BG_+)$.  The periodic $K$-theory $KU^* (BG_+) $ of the finite group $G$ is known by the Atiyah-Segal completion theorem to be trivial in odd degrees and isomorphic in even degrees to the completion  $R(G) ^{\wedge} _I$, where $I$ is the augmentation ideal of the
complex representation ring $R (G)$.

Here we restrict to the case of elementary abelian $2$-groups, and $V
\mapsto ku^* (BV_+)$ is considered as a contravariant functor of $V \in \obj \fdvs$. The result is proved by applying  Proposition \ref{prop:qfrak_formal}, for which an understanding of $H\zed^* (BV_+) $ and the action of $\qfrak$ is required.

\begin{prop}
\label{prop:HZ_cohom_BV}
 There are natural isomorphisms:
\begin{eqnarray*}
 H \zed ^n (BV_+) 
&\cong&
 \left\{
\begin{array}{ll}
\zed & n= 0 \\
K_n (V^\sharp) & n >0 
\end{array}
\right.
\\
(\mathrm{Im} \qfrak)^n &\cong& L_n (V^\sharp) \\
(\mathrm{Ker} \qfrak)^n & \cong &  \left\{
\begin{array}{ll}
\zed & n= 0 \\
\tilde{L}_n (V^\sharp) & n >0. 
\end{array}
\right.
\end{eqnarray*}
\end{prop}

\begin{proof}
 The algebra $H\field^* (BV_+)$ is naturally isomorphic to $S^* (V^\sharp)$ and integral reduced cohomology  $H\zed^* (BV)$ embeds in $H\field^* (BV)$ as the kernel of the Bockstein operator. Hence (paying attention to the behaviour in unreduced cohomology),  the result follows from Lemma \ref{lem:Bockstein_exact_Q1_qfrak}, using the definition of the functors $K_n$, $L_n$ and $\tilde{L}_n$ from Section \ref{sect:milnor}.
\end{proof}

\begin{lem}
\label{lem:T_degree_one}
 There are identities $\torsion ku^* (B0_+) = 0 = \torsion ku^* (B\zed/2_+) $.  
\end{lem}

\begin{proof}
 This follows  from the identification of $ku^* (B\zed/2_+)$ using the Gysin sequence (cf. \cite[Section
2.2]{bruner_greenlees}).
\end{proof}

\begin{thm}
\label{thm:ku_cohom_BV}
For $V \in \obj \fdvs$, there are natural isomorphisms
\begin{eqnarray*}
 \torsion ku^n (BV_+) &\cong& \annih ku^n (BV_+) \cong  (\mathrm{Im} \qfrak)^n \\
ku^n (BV_+)/ v & \cong & (\mathrm{Ker} \qfrak) ^n 
\end{eqnarray*}
and  $\torsion ku^n (BV_+ ) \rightarrow
ku^n (BV_+) / v$  is an isomorphism for $n$ odd and, for $n =2d>0$, there is a natural short exact sequence:
\[
 0 \rightarrow 
\torsion ku^{2d} (BV)
\rightarrow ku^{2d} (BV) / v
\rightarrow 
p_d \Ibar (V^\sharp) 
\rightarrow 
0.
\]
The surjection $\cotorsion ku^* (BV_+ ) \twoheadrightarrow \cotorsion ku^* (BV_+) /v $ induces a pullback diagram of short exact
sequences
\[
 \xymatrix{
0 
\ar[r]
&
\torsion ku^*  (BV_+) 
\ar[r]
\ar@{=}[d]
&
ku^{*}(BV_+) 
\ar[r]
\ar[d]
&
\cotorsion ku^{*} (BV_+) 
\ar@{->>}[d]
\ar[r]
&
0
\\
0 
\ar[r]
&
\mathrm{Im} \qfrak
\ar[r]
&
\mathrm{Ker} \qfrak
\ar[r]
&
 (\cotorsion ku^{*} (BV_+)) /v
\ar[r]
&
0.
}
\]
There are natural isomorphisms $\cotorsion ku^* (BV_+) \cong \zed[v] \oplus \cotorsion ku^* (BV)$ and 
\[
 \cotorsion ku^{2d}(BV) \cong \Pbarzedtwo^d (V^\sharp)
\]
inducing an isomorphism of short exact sequences (for $d>0$):
\[
 \xymatrix{
0 
\ar[r]
&
\cotorsion ku^{2d+2} (BV) 
\ar[d]_\cong
\ar[r]^{v}
&
\cotorsion ku^{2d} (BV) 
\ar[d]_\cong
\ar[r]
&
p_d \Ibar (V^\sharp)
\ar@{=}[d]
\ar[r]
&
0
\\
0
\ar[r]
&
\Pbarzedtwo^{d+1} (V^\sharp)
\ar[r]
&
\Pbarzedtwo^d (V^\sharp)
\ar[r]
&
p_d \Ibar  (V^\sharp)
\ar[r]
&
0.
}
\]
\end{thm}

\begin{proof}
 The first part of the Theorem  follows from Proposition \ref{prop:qfrak_formal},  since Lemma \ref{lem:subfunctor_pdI_qdP} implies that the cohomological hypothesis is satisfied. To apply the Proposition, it is sufficient to show that $\mathrm{Im} \qfrak \cong \big(\torsion ku^{*} (BV_+) \big)/v$. 

By  Proposition \ref{prop:HZ_cohom_BV} and Proposition \ref{prop:homology_K_Q1}, 
\[
 (\mathrm{Ker} \qfrak / \mathrm{Im} \qfrak)^n
 \cong 
\left\{
\begin{array}{ll}
 \zed & n=0 \\
0 & n \ \mathrm{odd}\\
p_d \Ibar (V^\sharp) & n =2d, \ d>0.
\end{array}
\right.
\]
In odd degrees $\mathrm{Im} \qfrak = \mathrm{Ker} \qfrak$, and the result follows from the inclusions given in Lemma \ref{lem:ses_tors_cotors}.

It remains to show that the inclusion $ (\mathrm{Im}\qfrak)^{2d}  \hookrightarrow \torsion ku^{2d} (BV_+) /v$ is an isomorphism for $d \in \nat$. For $d=0$, both terms are zero; for $d>0$,  the cokernel is a  subfunctor of $V \mapsto p_d \Ibar(V ^\sharp)$ by the  inclusions given in Lemma \ref{lem:ses_tors_cotors} and the above identification of $\mathrm{Ker} \qfrak / \mathrm{Im} \qfrak$, hence it suffices to show that the cokernel is trivial when evaluated on $V=\field$, by Lemma \ref{lem:subfunctor_pdI_qdP}. This follows  from the fact that  $\torsion ku^* (B\zed/2_+)=0$, by Lemma \ref{lem:T_degree_one}.

Finally, consider $\cotorsion ku^* (BV_+)$.  For $d=0$, the result is the Atiyah-Segal completion theorem;  the structure in higher degree follows by induction on $d$ from the results of Section \ref{subsect:Pbarzed}, in particular Proposition \ref{prop:subquotients_Pbarzed}, using the identification of the functors $$(\cotorsion ku^* (BV_+))/v \cong \mathrm{Ker} \qfrak / \mathrm{Im} \qfrak$$
 given above.
\end{proof}

\begin{rem}
Proposition \ref{prop:mult_ku_cohom} applies to $ku^* (BV_+)$ to give a description of the algebra structure (cf. \cite{bruner_greenlees}).
\end{rem}

\subsection{The $ku$-homology of elementary abelian $2$-groups}

The $ku$-homology of elementary abelian $2$-groups can be determined as for $ku$-cohomology, by applying Proposition \ref{prop:qfrak_formal}, for which an understanding of $H\zed_* (BV_+) $ and the action of $\qfrak$ is required.

\begin{prop}
\label{prop:HZ_hom_BV_ qfrak}
There are natural isomorphisms: 
\begin{eqnarray*}
 H\zed_n (BV_+) 
&\cong &
\left\{
\begin{array}{ll}
\zed & n = 0\\
DK_{n+1} (V) & n >0 
\end{array}
\right.
\\
(\mathrm{Im} \qfrak)_n & \cong & D L_{n+4} \\
(\mathrm{Ker} \qfrak)_n & \cong & 
\left\{
\begin{array}{ll}
\zed & n=0\\
 D(K_{n+1}/L_{n+1})(V)& n >0.
\end{array}
\right.
\end{eqnarray*}

Moreover:
\[
(\mathrm{Ker} \qfrak / \mathrm{Im} \qfrak ) _n \cong 
 \left\{\begin{array}{ll}
    \zed & n =0 \\
0 & n \equiv 0 \mod 2 \\
q_d \Pbar (V) & n = 2d-1 >0.    
        \end{array}
\right.
\]
\end{prop}

\begin{proof}
There is a  natural isomorphism $H\field_n (BV_+) \cong \Gamma^n (V)$ and  $ H \zed \stackrel{\rho}{\rightarrow} H\field $ induces a monomorphism in reduced homology, so Lemma \ref{lem:Bockstein_exact_Q1_qfrak} applies. It remains to identify the functors upon dualizing, with the attendant shift in indexing. 

There is a natural isomorphism 
\[
 H\zed_n (BV_+) \cong D K_{n+1} (V),
\]
  for $n>0$ and the natural transformation $\qfrak$ is induced by $DQ_1$.

Consider the $Q_1$ complex $
 K_{n-3} \rightarrow K_{n} \rightarrow K_{n+3}
$ (for $n >3$); this gives rise to short exact sequences
\begin{eqnarray*}
&&0 \rightarrow L_n \rightarrow \tilde{L}_n\rightarrow H_n \rightarrow 0\\
&&0 \rightarrow \tilde{L}_n \rightarrow K_n \rightarrow L_{n+3} \rightarrow 0\\
&&0 \rightarrow H_n \rightarrow K_n / L_n \rightarrow L_{n+3} \rightarrow 0,
\end{eqnarray*}
where $H_n$ denotes the homology. 

In the dual complex, $ DK_{n+3} \stackrel{DQ_1}{\rightarrow} DK_{n} \stackrel{DQ_1}{\rightarrow} DK_{n-3}
$:
\begin{eqnarray*}
 \mathrm{Im}{DQ_1} &\cong& DL_{n+3}\\
\mathrm{Ker} {DQ_1} &\cong& D(K_n /L_n),  
\end{eqnarray*}
the inclusion $ \mathrm{Im}{DQ_1} \hookrightarrow \mathrm{Ker} {DQ_1}$ is  dual to $K_n /L_n
\twoheadrightarrow L_{n+3}$ and the cokernel is $DH_n$. This proves the first statement (taking into account the unreduced homology in degree zero and the degree shift); the calculations of Example \ref{exam:L_low_degree} show that the expressions are correct in low degrees.

The final statement follows from Proposition \ref{prop:homology_K_Q1}, by dualizing.
\end{proof}

\begin{thm}
\label{thm:ku_homology_BV}
For $V \in \obj \fdvs$, there are natural isomorphisms:
\begin{eqnarray*}
\torsion ku_n (BV_+)& \cong& \annih ku_n (BV_+) \cong (\mathrm{Im} \qfrak)_n 
\\
ku_n (BV_+) /v& \cong& (\mathrm{Ker}\qfrak)_n 
\end{eqnarray*}
and the inclusion $\torsion ku_* (BV_+) \hookrightarrow ku_* (BV_+) $ induces a natural short exact sequence 
\[
 0 \rightarrow \torsion ku_* (BV_+) \rightarrow ku_* (BV_+) / v \rightarrow \cotorsion ku_* (BV_+) /v \rightarrow 0,
\]
where
\[
 \cotorsion ku_n (BV_+) / v \cong \left\{
\begin{array}{ll}
 \zed & n= 0 \\
0 & 0 < n \equiv 0 \mod 2\\
q_d \Pbar (V) & n = 2d -1 >0.
\end{array}
\right.
\]
There is a pullback diagram of short exact sequences 
\[
 \xymatrix{
0
\ar[r]
&
\torsion ku_* (BV_+) 
\ar[r]
\ar[d]_\cong
&
ku_* (BV_+) 
\ar[r]
\ar@{->>}[d]
&
\cotorsion ku_* (BV_+) 
\ar[r]
\ar@{->>}[d]
&
0
\\
0\ar[r]
&
\mathrm{Im} \qfrak
\ar[r]
&
\mathrm{Ker} \qfrak
\ar[r]
&
( \cotorsion ku_* (BV_+ ) )/v 
\ar[r]
&
0.
}
\]

For $d>0$ an integer, there is a natural isomorphism 
\[
 \cotorsion ku_{2d-1}(BV) 
\cong 
R^d _\zed (V)
\]
inducing an  isomorphism of short exact sequences 
\[
 \xymatrix{
0\ar[r]
&
\cotorsion ku_{2d-1} (BV) 
\ar[d]_\cong
\ar[r]^{\times v}
&
\cotorsion ku_{2d+1} (BV) 
\ar[d]^\cong
\ar[r]
&
q_{d+1} \Pbar (V)
\ar@{=}[d]
\ar[r]
&
0
\\
0
\ar[r]
&
R^d _\zed (V) 
\ar[r]
&
R^{d+1}_\zed (V)
\ar[r]
&
q_{d+1}\Pbar (V) 
\ar[r]
&
0,
}
\]
where $R^d_\zed \hookrightarrow R^{d+1}_\zed$ is the dual of the natural projection $R^{d+1}_\zed \twoheadrightarrow R^d
_\zed$.
\end{thm}

\begin{proof}
 There are natural monomorphisms 
\[
 \mathrm{Im}\qfrak (V)
\hookrightarrow 
ku_* (BV_+) /v
\hookrightarrow
\mathrm{Ker}\qfrak(V);
\]
by Proposition \ref{prop:HZ_hom_BV_ qfrak},  in positive even degree, these are isomorphisms;  in degree $n= 2d-1>0$, the quotient
$(\mathrm{Ker}\qfrak/ 
\mathrm{Im}\qfrak)(V) $ is $q_d \Pbar (V)$, hence there is a natural inclusion 
\[
 \big( (
ku_* (BV_+)/v 
)
/ 
 \mathrm{Im}\qfrak (V)
\big)_{2d-1}
\hookrightarrow 
q_d \Pbar (V).
\]
To prove the result, by Proposition \ref{prop:qfrak_formal},  if suffices to show that this is an isomorphism; hence, by Lemma \ref{lem:subfunctor_pdI_qdP}, it suffices to show that the left hand side is non-trivial when evaluated on $\field$, for all $d>0$. It is straightforward to verify that $ \mathrm{Im}\qfrak_{\mathrm{odd}} (\field)=0$, thus it suffices to show that $ku_* (B\zed/2_+)/v$ is non-trivial in all odd degrees, which follows from the structure of $ku_* (B\zed/2_+)$ as a graded abelian group (see
\cite[Section 3.4]{bruner_greenlees}, for example).

Finally, the identification of $\cotorsion ku_* (BV_+)$ follows from the results of Section \ref{subsect:R_functors}, in particular the short exact sequences of Lemma \ref{lem:Q_basis_properties}, the self-duality of the functors $R^n_\zed$ (Theorem \ref{thm:Q-self-dual}) and the Pontrjagin duality between $R^\infty_\zed$ and $\Pbarzedtwo$ (Theorem \ref{thm:pontrjagin_duality}). (Compare the proof of Theorem \ref{thm:ku_cohom_BV}.)
\end{proof}

\begin{rem}
By \cite[Proposition 3.2.1]{bruner_greenlees}, the
universal coefficient spectral sequence calculating $ku^* (BV)$ from $ku_*(BV)$ collapses to the short exact sequence 
\begin{eqnarray*}
 0
\rightarrow 
\ext^2_{ku_*} (\Sigma^2 \torsion ku_*(BV), ku_*) 
\rightarrow 
ku^* (BV) 
\rightarrow \\
\ext^1_{ku_*} (\Sigma ^1 \cotorsion ku_*(BV), ku_*)
\rightarrow 
0.
 \end{eqnarray*}

This is isomorphic to 
\[
 0
\rightarrow 
\torsion ku^* (BV) 
\rightarrow 
ku^* (BV) 
\rightarrow 
\cotorsion ku^* (BV) 
\rightarrow 
0
\]
and explains the  duality between $ku$-homology and $ku$-cohomology of $BV$. The analysis of \cite[Section 4.12]{bruner_greenlees} can be made functorial to give the identification of $\cotorsion ku_* (BV_+)$. 
\end{rem}

\section{Local duality}
\label{sect:local_duality}

An equivariant version of local duality (with respect to the action of the general linear groups $\mathrm{Aut}(V)$) is given as it
arises in the current context; this gives a refinement of the results of \cite[Section 4.7]{bruner_greenlees}.

\subsection{Categories of $S^\bullet$-modules}

Throughout this section, the prime $p$ is arbitrary.  The fact that the   functor $S ^\bullet $ takes
values in graded vector spaces of finite type will be used without further comment.

\begin{defn}\
\begin{enumerate}
\item
Let $\smodf$ denote the category of graded right $S^\bullet$-modules in $\f$ and
$S^\bullet$-module morphisms.
\item
For $V \in \obj \fdvs$,  let $\smodaut$ denote the category of graded right $S^\bullet
(V)$-modules in $\aut(V)$-modules and $S^\bullet (V)$-module morphisms.
\end{enumerate}
\end{defn}

\begin{rem}
The choice to work with right modules is dictated by  the notation adopted for Koszul complexes.
\end{rem}

An object of $\smodf$ is a graded functor $M^\bullet$, equipped with a structure
morphism $ M^\bullet \otimes S^\bullet \rightarrow M^\bullet$  which is unital
and associative; this can be expressed in terms of the components $M
^b \otimes S^a   \rightarrow M ^{a+b}$. A similar description holds for $S^\bullet
(V)$-modules.

\begin{lem}
\label{lem:S-modules}
For $V \in \obj \fdvs$,  the categories
$\smodf$ and $\smodaut$ are tensor abelian. Moreover:
\begin{enumerate}
	\item
	the forgetful functor $\smodf \rightarrow \f$ is exact and admits an
exact left adjoint $- \otimes S^\bullet  : \f \rightarrow \smodf$ which is
monoidal;
	\item
	the forgetful functor $\smodaut \rightarrow \aut(V)\dash\modules$ is exact
and admits an exact left adjoint $ - \otimes S^\bullet (V) : \aut(V)\dash\modules
\rightarrow \smodaut$ which is monoidal;
	\item
	evaluation at $V$,  $\f \rightarrow \aut(V)\dash\modules$,   induces an
exact tensor functor $\smodf \rightarrow \smodaut$.
\end{enumerate}
\end{lem}

\begin{proof}
	Clear.
\end{proof}

\begin{defn}
For $N \in \obj \smodf$, let $\smodhom (-, N) $ be the functor $(\smodf)\op \rightarrow \smodaut$
defined by
	\[
		\smodhom (M, N) := \hom _{S^\bullet (V)}(M (V) , N (V))
	\]
where the right hand side is equipped with the usual grading and  $\aut(V)$ acts
via conjugation.
\end{defn}

\begin{rem}
The above definition can be refined to give a coefficient system for the general linear groups over $\field$ associated to the pair $M,N$ of graded $S^\bullet$-modules in $\f$.
\end{rem}

\begin{lem}
\label{lem:free_adjunction}
	For $F \in \obj \f$ and $V \in \obj \fdvs$, there is a natural
isomorphism:
	\[
		\smodhom ( F \otimes S^\bullet , S^\bullet) \cong  F (V) ^\sharp \otimes S^\bullet(V)
	\]
	in $\aut(V)\dash\modules$, where $F(V) ^\sharp$ is equipped with the  contragredient action.
\end{lem}

\begin{proof}
	Straightforward.
\end{proof}

The following is clear:

\begin{lem}
\label{lem:morphism_smodhom}
	Let $F, G \in \obj \f$ and $\alpha :  F \otimes S^\bullet \rightarrow
 G \otimes S^\bullet$ be a morphism of $\smodf$, induced by $ \tilde{\alpha} : F \rightarrow
 G \otimes S^\bullet$ in $\f$.

Then $\smodhom (\alpha, S^\bullet)$ identifies with the
morphism $ G(V)^\sharp \otimes S^\bullet (V)\rightarrow F(V) ^\sharp\otimes S^\bullet (V)$ of $\smodaut$ induced by  $\gamma : G(V) ^\sharp
\rightarrow  F(V) ^\sharp \otimes S^\bullet (V)$ in $\aut(V)\dash\modules$, where 
		$\gamma$ is adjoint (in the category
$\aut(V)\dash\modules$) to the morphism $G(V)^\sharp \otimes S^\bullet (V)
^\sharp  \rightarrow F(V)^\sharp$ dual to the evaluation  of
$\tilde{\alpha}$ on $V$.
\end{lem}

\subsection{The dualizing functor}

Recall that the exterior power functors are self-dual, so that:
 
\begin{lem}
 \label{lem:Lambda_duality}
For $n \in \nat$,  there is a natural isomorphism of contravariant
functors of $V$:
$
 \Lambda ^n (V^\sharp) \cong \Lambda^n (V) ^\sharp.
$
\end{lem}

This is combined with the following duality result, when restricting to the
consideration of the $\aut (V)$-action:

\begin{lem}
\label{lem:Lambda_twisted_iso}
 Let  $0 \leq j \leq r$ be  integers and $V \in \obj \fdvs$ have rank $r$. The
composite
 \[
 	\Lambda^r (V) \otimes \Lambda^j (V) ^\sharp
 	\stackrel{\Delta \otimes 1}{\rightarrow}
 	\Lambda^{r-j} (V) \otimes \Lambda^j (V) \otimes \Lambda^j (V)^\sharp
 	\rightarrow
 	\Lambda ^{r-j} (V),
 \]
where the second morphism is induced by evaluation $\Lambda^j (V) \otimes
\Lambda^j (V) ^ \sharp \rightarrow \field$, induces an isomorphism $\Lambda^r (V) \otimes \Lambda^j (V)^ \sharp \cong \Lambda
^{r-j} (V)$ in $\aut(V)\dash\modules$, where $\Lambda^j (V)^\sharp$ is equipped with the contragredient
$\aut(V)$-module structure.
\end{lem}

\begin{proof}
 The result follows from the fact that the  product $\Lambda^{r-j} (V)
\otimes \Lambda^j (V) \rightarrow \Lambda^r (V) \cong \field$ defines a perfect
pairing and the equivariance of the evaluation map. 
\end{proof}

\begin{lem}
\label{lem:coproduct_local_duality}
	Let  $1 \leq j \leq r$ be integers, $V \in \obj \fdvs$ have rank $r$ 
and write $\mu : \Lambda^{j-1} (V) ^\sharp \otimes \Lambda^1 (V) ^\sharp
\rightarrow \Lambda^j (V)^{\sharp}$ for the product morphism (dual to the
evaluation of the coproduct $\Lambda^j \rightarrow \Lambda^{j-1} \otimes
\Lambda^1$).

Then, under the isomorphism of Lemma \ref{lem:Lambda_twisted_iso},
$\Lambda^r (V) \otimes \mu$ is $\aut(V)$-equivariantly isomorphic to the
morphism
	\[
		\Lambda^{r-j+1} (V) \otimes \Lambda^{1}(V) ^\sharp \rightarrow
\Lambda^{r-j} (V)
	\]
which is adjoint to the evaluation of the coproduct $\Lambda^{r-j+1} \rightarrow \Lambda^{r-j} \otimes \Lambda^1$ on $V$.
\end{lem}

\begin{proof}
A consequence of the coassociativity of the comultiplication on
the exterior power functors and the fact that the multiplication $\mu$ is dual
to the coproduct 
$\Lambda^j \rightarrow \Lambda^{j-1} \otimes \Lambda^1$.
\end{proof}

\begin{nota}
	For $i \in \nat$  and $V \in \obj \fdvs$ of rank $r$, let
\begin{enumerate}
	\item
	$\tau_i: \Lambda^1 \rightarrow S^{p^i}$ denote the composite of the
isomorphism $\Lambda^1 \cong S^1$ with the iterated Frobenius $S^1
\hookrightarrow S^{p^i}$ and also the induced Koszul
differential
	\[
		\tau_i : \Lambda^j  \otimes S^\bullet \rightarrow \Lambda^{j-1} \otimes S^{\bullet
+p^i} 
	\]
in the category $\smodf$,  induced by the composite morphism:
\[
	\Lambda^j \stackrel{\Delta}{\rightarrow} \Lambda^{j-1} \otimes \Lambda^1
\stackrel{1 \otimes \tau_i}{\rightarrow} \Lambda^{j-1} \otimes S^{p^i};
\]
\item
$\kz_i$ denote the Koszul complex in $\smodf$:
\[
	\ldots \rightarrow  \Lambda^j \otimes S^\bullet 
\stackrel{\tau_i}{\rightarrow} \Lambda^{j-1} \otimes  S^{\bullet +p^i} \rightarrow
\ldots \rightarrow S^{\bullet + jp^i}
	\rightarrow 0;
\]
\item
$\kz_i(V)$ denote  the Koszul complex in $\smodaut$:
\begin{eqnarray*}
	0 \rightarrow
	\Lambda^r (V) \otimes  S^{\bullet}(V) 
	\rightarrow \ldots
\rightarrow
 \Lambda^j(V) \otimes S^{\bullet+ (r-j)p^i}(V)
	\\
	\stackrel{\tau_i}{\rightarrow}
	\Lambda^{j-1}(V) \otimes  S^{\bullet +(r-j+1)p^i}(V) 
	\rightarrow \ldots \rightarrow S^{\bullet + rp^i}(V)\rightarrow
0.
\end{eqnarray*}
\end{enumerate}
\end{nota}

\begin{prop}
\label{prop:dual_tau}
	For integers  $1 \leq j \leq r$, $V \in \obj \fdvs$ of rank $r$ 
and $i\in \nat$,  the morphism
	\[
		\smodhom (\tau_i, S^\bullet) : \smodhom (
\Lambda^{j-1} \otimes S^\bullet , S^\bullet)
		\rightarrow
		\smodhom (\Lambda^{j} \otimes S^\bullet , S^\bullet)
	\]
is induced by the morphism $\gamma : \Lambda^{j-1}(V) ^\sharp \rightarrow
\Lambda ^j (V) ^\sharp \otimes S^{p^i}(V)$ such that, under the isomorphism of
Lemma \ref{lem:Lambda_twisted_iso},  $\Lambda^r (V) \otimes \gamma :
\Lambda^{r-j+1} (V) \rightarrow  \Lambda^{r-j} (V) \otimes S^{p^i}(V)$ is
$\aut(V)$-equivariantly isomorphic to the evaluation on $V$ of the Koszul differential.
\end{prop}

\begin{proof}
 Combine Lemma \ref{lem:morphism_smodhom} with Lemma \ref{lem:coproduct_local_duality}.
\end{proof}

\begin{cor}
	\label{cor:dualizing_functor}
	For $V \in \obj \fdvs$ of rank $r$ and $i\in \nat$, there
is a natural isomorphism of complexes:
	\[
		\smodhom (\kz_i ,  \Lambda^r \otimes S^\bullet)
		\cong
		\kz_i (V)
	\]
in $\smodaut$.
\end{cor}

\begin{proof}
	After addition of the twisting functor $\Lambda^r$, the result is an
immediate consequence of Proposition \ref{prop:dual_tau}.
\end{proof}

\begin{rem}
	Taking $i=0$, so that $\kz_0$ is the usual Koszul complex, which has
homology $\field$ concentrated in homological degree zero, this shows that
$ \Lambda^r \otimes S^\bullet$ plays the role of the dualizing object,
corresponding to the fact that $S^\bullet$ is graded Gorenstein (cf. \cite[Section 3.7]{bruns_herzog}).
\end{rem}

\subsection{Local cohomology in $\smodaut$}
In this section, local cohomology is considered with respect to the augmentation ideal $I$ of $S^\bullet (V)$, where $V \in \obj  \fdvs$ has rank $r$;
from the results of the previous section, it follows that  the  local duality isomorphism (cf. \cite{bruns_herzog})
should be interpreted as stating that the local cohomology of the
$S^\bullet (V)$-free object of $\smodaut$,  $  F (V) \otimes S^\bullet (V)$, for $F \in \obj  \f$,  is concentrated in cohomological degree $r$, where it
is isomorphic to
\[
	\hom_{\fdvs}\big(\smodhom ( F \otimes S^\bullet , \Lambda^r\otimes S^\bullet)
, \field \big).
\]

 \begin{rem}
 In the cases of interest here, $F$ takes finite-dimensional values, so that $\smodhom
( F \otimes S^\bullet,  \Lambda^r\otimes S^\bullet)$ is a graded vector space of
finite type.
\end{rem}

\begin{prop}
(Cf. \cite[Lemma 4.7.1]{bruner_greenlees}.)
\label{prop:local_cohomology_functors}
Let $V \in \obj \fdvs$ have rank $r$ and $G \in \obj \smodf$. Suppose that there exists a complex
	\[
	0 \rightarrow 	F_r \otimes S^\bullet \rightarrow F_{r-1} \otimes S^{\bullet} \rightarrow \ldots \rightarrow F_0 \otimes S^\bullet \rightarrow G
	\rightarrow 0
	\]
which induces an exact sequence in $\smodaut$, after evaluation on $V$. Then the local cohomology of $G(V)  \in \obj \smodaut$ is
\[
	H^i _I (G (V) ) \cong H_{r-i} \big(\hom_{\fdvs} ( \smodhom (F_j \otimes S^\bullet,  \Lambda^r \otimes S^\bullet) , \field)\big), 
\]
up to shift in  grading.
\end{prop}

\section{Filtering symmetric powers  and Koszul complexes}

To calculate the local cohomology of $\torsion ku^* (BV_+)$ at the prime two by
using a form of local duality, it is necessary to filter and study the
associated Koszul complexes (for odd primes, this filtration step is
unnecessary). The results of this  section refine those of  Section \ref{sect:milnor}.

\subsection{Filtering the symmetric powers}
\label{subsect:filter_symm_powers}

Let $\Phi S^\bullet $ denote the image of the Frobenius, so that  $\Phi
S^\bullet$ takes values in graded commutative algebras and the canonical
morphisms $S^\bullet \stackrel{\cong}{\rightarrow} \Phi S^\bullet
\hookrightarrow S^\bullet$ are morphisms of algebras. 

\begin{defn}
For $t\in \nat$,  let $f_t S^\bullet \subset S^\bullet$ denote the
image of multiplication
\[
 S^t \otimes \Phi S^\bullet  \stackrel{\mu} {\rightarrow} S^\bullet.
\
\]
\end{defn}

\begin{lem}
\label{lem:filt_S}
For  $t \in \nat$,
\begin{enumerate}
 \item 
the functors $f_t S^\bullet \subset S^\bullet$ define an increasing filtration
of $S^\bullet$:
\[
 f_0 S^\bullet = \Phi S^\bullet \subset f_1 S^\bullet \subset f_2 S^\bullet
\subset \ldots \subset f_t S^\bullet \subset \ldots \subset S^\bullet
\]
in  $\Phi S^\bullet$-modules;
\item
there is an isomorphism of $S^\bullet$-modules:
\[
 f_t S^\bullet / f_{t-1} S^\bullet \cong \Lambda ^t \otimes S^\bullet,
\]
where $S^\bullet$ acts on the left hand side by restriction along $S^\bullet
\stackrel{\Phi \ \cong}{\rightarrow} \Phi S^\bullet$ and  by multiplication on
the right hand factor of $\Lambda ^t \otimes S^\bullet$; for $V \in\obj \fdvs$, this restricts to an isomorphism of $S^\bullet (V)$-modules:
\[
 \bigoplus _{t \geq 0} \big( f_t S^\bullet / f_{t-1} S^\bullet \big) (V) 
\cong 
S^\bullet (V), 
\]
where $S^\bullet (V)$ acts on the right hand side via $\Phi$; 

\item
for $V \in \obj \fdvs$, the  inclusion 
$
 f_t S^\bullet (V) \hookrightarrow S^\bullet (V)
$
is an isomorphism for $t \geq \dim V$.
\end{enumerate}
\end{lem}

\begin{proof}
 Straightforward; to prove that $S^\bullet (V)$ is isomorphic to $\bigoplus _{t \geq 0} \big( f_t S^\bullet / f_{t-1} S^\bullet \big) (V)$ as $S^\bullet (V)$-modules, it is sufficient to consider monomial bases. (Note that this statement is not true $\aut(V)$-equivariantly.)
\end{proof}

For notational clarity, shifts in gradings are omitted from the following
statement.

\begin{prop}
\label{prop:Q_filter_tau}
 For $i, t \in \nat$, the Milnor derivation $Q_i : S^\bullet \rightarrow
S^{\bullet} $ is a morphism of $\Phi S^\bullet$-modules and restricts to 
 $
 f_t S^\bullet \stackrel{f_t Q_i} {\rightarrow} f_{t-1} S^\bullet;
$
the induced morphism  on the filtration quotients 
\[
\Lambda^t \otimes S^\bullet   \cong  f_t S^\bullet / f_{t-1} S^\bullet
\rightarrow 
f_{t-1} S^\bullet / f_{t-2} S^\bullet \cong \Lambda^{t-1} \otimes S^\bullet
\]
is the Koszul differential $\tau_i$.
\end{prop}

\begin{proof}
 Straightforward.
\end{proof}

The  differentials $\tau_0$ and $\tau_1$ define a bicomplex structure on
$\Lambda^\bullet \otimes S^\bullet$, which  can be displayed as:
\[
\xymatrixrowsep{10pt}
\xymatrixcolsep{10pt}
 \xymatrix{
&&&&&&&&\field\\
&&&&&&\Lambda ^1 \ar[r] \ar[d]& S^1 \\
&&&& \Lambda^2 \ar[r]\ar[d] & \Lambda  ^1 \otimes S^1 \ar[r]\ar[d] & S ^2 \\
&&\Lambda ^3 \ar[r]\ar[d]& \Lambda^2 \otimes S^1 \ar[r]\ar[d]& \Lambda^1 \otimes
S^2 \ar[r]\ar[d] & S^3 \\
\Lambda^4 \ar[r] \ar@{.>}[d]& \Lambda^3 \otimes S^1 \ar[r]\ar@{.>}[d] &
\Lambda^2 \otimes S^2 \ar[r]\ar@{.>}[d]& \Lambda^1 \otimes S^3
\ar[r]\ar@{.>}[d]& S^4 \\
&&&&&&&
}
\]

The portion of the bicomplex displayed indicates the essential features:
\begin{enumerate}
 \item 
the rows (respectively columns) are Koszul complexes, with differential $\tau_0$
(resp. $\tau_1$);
\item
the bicomplex is concentrated in a single quadrant and there are vanishing lines
of slope $1/2$ and $1$.
\end{enumerate}

\subsection{Filtering the functors $K_n$}

This section establishes the filtered version of Proposition
\ref{prop:homology_K_Q1}; the starting point is  the functorial homology of the
Koszul complexes, which is a standard calculation,  related to Proposition
\ref{prop:homology_Qi}:

\begin{prop}
\label{prop:hom_Koszul_complex}
	For $i \in \nat$, the homology of $(\Lambda^\bullet \otimes
S^\bullet , \tau_i)$ is $S^\bullet / \langle x^{2^i} \rangle$, concentrated in
homological degree zero.
\end{prop}

The following are analogous to the functors $K_n$ introduced in
Notation \ref{nota:Kn}:

\begin{nota}
 For an integer $a\geq1$ and $b \in \nat$, let $\kfrak_{a,b}$ denote the image of $\tau_0 :
\Lambda^a \otimes S^b \rightarrow \Lambda^{a-1} \otimes S^{b+1}$ and, by convention:
 \[
\kfrak_{0,b}:= \left\{
\begin{array}{ll}
\field & b=0\\
0 & b >0.
\end{array}
\right.
 \]
\end{nota}

\begin{rem}
 \ 
\begin{enumerate}
\item
The functor $\kfrak_{1,b}$ identifies with the symmetric power $S^{b+1}$.
 \item 
Taking the image of $\tau_0$ rather than the kernel, is best suited for the current application, where  the homology of the Koszul
complex intervenes.
\end{enumerate}
\end{rem}

\begin{lem}
\label{lem:K_associated_graded}
For $0<n \in \zed$, the filtration $f_* S^\bullet$ induces a filtration of
$K_n$ with associated graded:
\[
\mathrm{gr}K_n \cong  \bigoplus _{a+	2b+1 = n} \kfrak_{a,b}.
\]
\end{lem}

\begin{proof}
	For $n>0$, $K_n$ is the image of $Q_0 : S^{n-1} \rightarrow S^n$.
Passing to the associated graded, the morphism $Q_0$ induces
	\[
		\bigoplus \tau_0 : \bigoplus _{a+2b+1=n} \Lambda^a \otimes S^b
\rightarrow \bigoplus _{a+2b+1=n}\Lambda^{a-1} \otimes S^{b+1},
	\]
by Proposition \ref{prop:Q_filter_tau}. Moreover, evaluated on $V \in \obj
\fdvs$, as a morphism of vector spaces, $Q_0$ identifies with $\bigoplus
\tau_0$, using the splitting of the filtration in $S^\bullet (V)$-modules given in Lemma \ref{lem:filt_S}.
\end{proof}

\begin{lem}
\label{lem:tau1_diff_J}
 The derivation $\tau_1$ induces a differential $\tau_1 : \kfrak_{a+1, b-2}
\rightarrow \kfrak_{a,b}$. For $a>0$ and $b \geq 0$, the short exact sequences
from the Koszul complexes:
\[
 0 
\rightarrow 
\kfrak_{a,b} \rightarrow \Lambda^{a-1} \otimes S^{b+1}
\rightarrow 
\kfrak_{a-1,b+1}
\rightarrow 0
\]
induce a short exact sequence of complexes:
\[
 \xymatrix{
&\ar@{.>}[d]& \ar@{.>}[d]& \ar@{.>}[d] \\
0\ar[r]&
\kfrak_{3, b-4} \ar[r] \ar[d]_{\tau_1}& \Lambda^2 \otimes S^{b-3}
\ar[r]\ar[d]_{\tau_1}& \kfrak_{2, b-3} \ar[d]_{\tau_1}
\ar[r]&0
\\
0\ar[r]&
\kfrak_{2, b-2} \ar[r] \ar[d]_{\tau_1}& \Lambda^1 \otimes S^{b-1}
\ar[r]\ar[d]_{\tau_1}& \kfrak_{1, b-1} \ar[d]_{\tau_1}
\ar[r]&0
\\
0\ar[r]&
\kfrak_{1,b} \ar[r]  & S^{b+1} \ar[r] & 0
\ar[r]&0.
}
\]
\end{lem}

\begin{proof}
The horizontal $\tau_0$-Koszul complexes are acyclic.
\end{proof}

\begin{prop}
\label{prop:J-complex_homology}
For $b\in \nat$, the complex
\[ 
 \ldots \rightarrow \kfrak_{n+1, b-2n} \stackrel{\tau_1} {\rightarrow}
\kfrak_{n, b-2n +2} \rightarrow \ldots
\rightarrow \kfrak_{1, b} \cong S^{b+1}
\]
 has homology $p_{b+1} \Ibar$ concentrated in homological degree zero.
\end{prop}

\begin{proof}
 The proof is by induction upon $b$, using the short exact sequence of complexes
provided by Lemma \ref{lem:tau1_diff_J}. The initial case $b=0$ is  by
inspection; for the inductive step, use the fact that the $\tau_1$ Koszul
complex:
\[
 \ldots \rightarrow \Lambda^{n} \otimes S^{b-2n+1} \stackrel{\tau_1}
{\rightarrow} \Lambda^{n-1} \otimes S^{b-2n+3} \rightarrow \ldots \rightarrow
S^{b+1}
\]
has homology $\Lambda^{b+1}$ concentrated in homological degree zero, by
Proposition \ref{prop:hom_Koszul_complex} (with $i=1$). The proof is completed
by the argument employed in the proof of Proposition \ref{prop:homology_K_Q1}.
\end{proof}

\subsection{Filtering the functors $L_n$ and $\tilde{L}_n$}

\begin{nota}
 For integers $a>0$, $b \geq 0$, let $\lfrak_{a,b}$ denote the cokernel of
$\tau_1 : \kfrak_{a+1, b-2} \rightarrow \kfrak_{a,b}$.
\end{nota}

Proposition \ref{prop:J-complex_homology} implies the following identification:

\begin{lem}
\label{lem:ker_ident_lfrak}
\
\begin{enumerate}
	\item
	 If $a>1$,  $\lfrak_{a,b} \cong \mathrm{Ker} \{  \kfrak_{a, b}
\stackrel{\tau_1}{\rightarrow} \kfrak_{a-1,b+2}\}$;
	 \item
	 $\lfrak_{1,b} \cong p_{b+1} \Ibar$. 
\end{enumerate}
\end{lem}

Recall from Section \ref{subsect:Q0_kernel} that $Q_1$ induces a morphism $Q_1 :
K_{n-3} \rightarrow K_n$ with image
$L_n$ and kernel $\tilde{L} _{n-3}$. It follows that the cokernel of $Q_1$
occurs in an extension
\[
 0
\rightarrow
\tilde{L}_n / L_n \rightarrow \mathrm{Coker}Q_1 \rightarrow L_{n+3} \rightarrow
0.
\]

\begin{lem}
\label{lem:associated_graded_L}
 For $n>0$ and $\mathrm{Coker}Q_1$ as above, the filtration $f_* S^\bullet$
induces a finite filtration of $\mathrm{Coker}Q_1$    with associated
graded
\[
\mathrm{gr}\mathrm{Coker}Q_1 \cong
 \bigoplus_{a+2b+1 =n, a \geq 1} \lfrak_{a,b},
\]
where, for $2 (b+1)=n$, the subobject $\lfrak_{1, b}$ is isomorphic to
$\tilde{L}_{2(b+1)}  /L_{2(b+1)} \cong p_{b+1} \Ibar $ and there is an induced
isomorphism:
\[
\mathrm{gr}L_{n+3} \cong
 \bigoplus_{a+2b+1 =n, a \geq 2} \lfrak_{a,b}.
\]
\end{lem}

\begin{proof}
	The result follows as for the proof of Lemma
\ref{lem:K_associated_graded}.
\end{proof}

\begin{rem}
 For the calculations of local cohomology, it is important that the isomorphisms of Lemma \ref{lem:K_associated_graded} and of Lemma \ref{lem:associated_graded_L} upon evaluation on $V \in \obj \fdvs$ correspond to isomorphisms in the category of $S^\bullet (V)$-modules (as in Lemma \ref{lem:filt_S}). 
\end{rem}

\section{The local cohomology spectral sequence}

The local cohomology theorem for $ku$ implies that there is a spectral sequence:
\[
 E^2 := H^{*,*} _I (ku^* (BV_+))
\Rightarrow 
ku_* (BV_+),
\]
where the $E^2$-term is the local cohomology with respect to the augmentation
ideal (see \cite{bruner_greenlees} for  generalities on the spectral
sequence, for arbitrary finite groups). Here $V$ is taken to be a fixed elementary abelian
$2$-group of rank $r$. The aim of this section is to indicate how the spectral sequence can be
understood  conceptually, by using the functorial calculations introduced
in Section \ref{sect:milnor}. 

 There are two key ingredients: the functorial description of local duality and of local cohomology
given in Section \ref{sect:local_duality} and an explanation of the relationship
between the local cohomology of $\torsion ku^* (BV_+) $ and   
that of $\cotorsion ku^* (BV_+)$.

The local cohomology spectral sequence can be made $\aut (V) $-equivariant but
it is clearly not functorial as it stands with respect to arbitrary vector space morphisms;
the techniques of this section do however show that the behaviour of the spectral
sequence is largely determined by functorial structure. 

\begin{rem}
Throughout,  the grading shifts  resulting from working with graded modules are suppressed. The gradings are not essential for the presentation of the arguments; the reader is encouraged to supply them.
\end{rem}

\subsection{The case of integral cohomology}

The local cohomology spectral sequence for $\hz_* (BV_+)$ already
illustrates some of the salient features of the local cohomology spectral sequence.
It can also be used in the analysis of the local cohomology spectral
sequence for $ku_* (BV_+)$ via the morphism induced by  $ku \rightarrow \hz$.

Let $V \in \obj \fdvs$ have rank $r$ and consider the short exact sequence
\[
	0
	\rightarrow
	\hz^* (BV)
	\rightarrow
	\hz ^* (BV_+)
	\rightarrow
	\zed
	\rightarrow
	0
\]
relating reduced and unreduced cohomology of $BV$, in the category of $\hz^*
(BV_+)$-modules, so that $\hz^* (BV)$ corresponds to the augmentation ideal $I$.
 There is an induced exact sequence of local cohomology groups:
 \begin{eqnarray}
  \label{eqn:connecting_hz}
0
 	\rightarrow
 	H^0_I (\hz^* (BV))
 	\rightarrow
 	H^0 _I (\hz^* (BV_+))
 	\rightarrow
\\ 
 \notag	
\rightarrow \zed
 	\rightarrow
 	H^1 _I (\hz^* (BV))
 	\rightarrow
 	H^1_I (\hz^* (BV_+))
 	\rightarrow 0
 \end{eqnarray}
and, for $j >1$, a natural  isomorphism $H^j _I (\hz^* (BV_+)) \cong H^j _I
(\hz^* (BV))$.

Hence, up to calculating the connecting morphism $\zed 	\rightarrow H^1 _I
(\hz^* (BV))$, the local cohomology of $\hz^* (BV_+)$ is determined by that of
$\hz^* (BV)$. Moreover,
it is clear that $H^1_I(\hz^* (BV))$ is annihilated by $2$, hence it suffices to
consider behaviour after reducing mod $2$.  

\begin{nota}
\label{nota:truncated_koszul}
	For $a \in \nat$, let $\sigma_{\geq a} \kz_0$ denote the brutal
truncation to the right of the Koszul complex:
	\[
		\ldots
		\rightarrow
		 \Lambda^{a+1}\otimes S^{\bullet-1} 
\stackrel{\tau_0}{\rightarrow}
 \Lambda^{a} \otimes S^\bullet
\]
and let $\sigma_{\leq a} \kz_0$ denote the brutal truncation to the left:
\[
	 \Lambda^{a} \otimes S^\bullet 
	\stackrel{\tau_0}{\rightarrow}
 \Lambda^{a-1}\otimes S^{\bullet +1} 
	\stackrel{\tau_0}{\rightarrow}
	\ldots
	\rightarrow
	S^{\bullet +a}.
\]
\end{nota}

\begin{lem}
\label{lem:local_cohom_calc_hz}
	For $a \in \nat$,
	\begin{enumerate}
		\item
		$\sigma_{\geq a}\kz_0 $ is an $S^\bullet$-free resolution of
$\kfrak_{a, \bullet}$;
		\item
		for $a \geq 1$, the complex $\sigma_{\leq a } \kz_0$ has
homology $\field$ in homological degree $0$ and $\kfrak_{a+1,\bullet}$ in homological
degree $a$.
		\end{enumerate}
Moreover, there are morphisms of complexes
\begin{eqnarray*}
	&& \kz_0 \twoheadrightarrow \sigma_{\geq 1 } \kz_0\twoheadrightarrow
\sigma_{\geq 2}\kz_0 \twoheadrightarrow \ldots \\
	&& \sigma_{\leq 0}\kz_0  \cong S^\bullet
\hookrightarrow \sigma_{\leq 1} \kz_0 \hookrightarrow \sigma_{\leq 2} \kz_0
\hookrightarrow \ldots \subset \kz_0 .
\end{eqnarray*}
\end{lem}

\begin{proof}
	Clear.
\end{proof}

\begin{prop}
\label{prop:local_duality_truncated_koszul}
	For $V \in \obj \fdvs$ of rank $r$ and integers $0\leq a \leq b  \leq r$, there
is a natural isomorphism of complexes:
	\[
		\smodhom (\sigma_{\geq a} \kz_0 , S^\bullet \otimes \Lambda^r)
		\cong
		\sigma_{\leq r-a} \kz_0 (V)
	\]
in $\smodaut$ and, with respect to these isomorphisms, the surjection
$\sigma_{\geq a} \kz _0 \twoheadrightarrow \sigma_{\geq b} \kz_0$ induces the 
inclusion
$ \sigma_{\leq r-b }\kz_0(V) \hookrightarrow \sigma_{\leq r-a} \kz_0(V)$.

In particular, the surjection $\kz_0 \twoheadrightarrow \sigma_{\geq b } \kz_0$
induces the inclusion:
\[
\sigma_{\leq r-b} \kz_0 (V) \hookrightarrow	\kz_0 (V),
\]
which induces an isomorphism in degree zero homology if $b < r$.
\end{prop}

\begin{proof}
	The result follows from Corollary \ref{cor:dualizing_functor}.
\end{proof}

\begin{rem}
It is useful to think of the surjection $\sigma_{\geq a} \kz_0 \twoheadrightarrow
\sigma_{\geq b}\kz_0$ as a morphism in an appropriate derived category
\[
	\kfrak_{a,\bullet}[a]
	\rightarrow
	\kfrak_{b,\bullet}[b],
	\]
where $[a], \ [b]$ correspond to  the shift in homological degree. In
particular, for $a=0$, this corresponds to $\field \rightarrow \kfrak_{b,\bullet}
[b]$.

The morphism $\sigma_{\geq a} \kz_0 \twoheadrightarrow \sigma_{\geq b}\kz_0$
induces a morphism between local cohomology groups (a generalized connecting morphism), via the
identification of local cohomology given in Proposition
\ref{prop:local_cohomology_functors} and  Proposition
\ref{prop:local_duality_truncated_koszul}.
\end{rem}

Using the above observation, one deduces:

\begin{lem}
\label{lem:connecting_hz_local_cohom}
	For $V \in \obj \fdvs$ of rank $r>1$, the connecting morphism
$H^0_I (\zed/2) \cong \zed/2 \rightarrow H^1 _I (\hz^* (BV))$ induced by the
short exact sequence of
	$\hz^* (BV_+)$-modules
	\[
	0
	\rightarrow
	\hz^* (BV)
	\rightarrow
	\hz ^* (BV_+)/2
	\rightarrow
	\zed/2
	\rightarrow
	0	
	\]
is non-trivial. 
\end{lem}

 \begin{rem}
\label{rem:E1_hz}
The connecting morphism in the long exact sequence for local cohomology, (\ref{eqn:connecting_hz}), is therefore non-trivial.  For a conceptual presentation of the results, it is useful to define an associated $E^1$-page, so that this connecting morphism appears as the $d^1$ differential.
 \end{rem}

The local cohomology ($r>1$) is as follows, using Lemma
\ref{lem:local_cohom_calc_hz}:
\[
	H^j _I (\hz^* (BV_+) ) \cong
	\left\{
\begin{array}{ll}
	\zed & j=0
\\
0 & j=1\\
\field & 2 \leq j \leq r-1\\
H^r _I (\hz^* (BV)) & j=r.
\end{array}
	\right.
\]
Moreover, $H^r _I (\hz^* (BV))$ has a finite filtration such that
\[
	\mathrm{gr} H^r _I (\hz^* (BV)) \cong
	\field
	\oplus
	\mathrm{gr} \hz _* (BV),
\]
up to shift in degree, where the filtration on homology $\hz_* (BV)$ is induced
by the filtration $f_t S^\bullet$. 

The analysis of the local cohomology spectral sequence is straightforward; the local cohomology corresponds to the $E^2$-page of the spectral sequence. The
permanent cycles in the zero column are given by the subgroup $2^{r-1} \zed$;
 the differentials $d^i$, for $2 \leq i \leq r$ are all non-trivial, starting from the zero column, and serve to
eliminate the extraneous factors of $\field$ which occur above.

\begin{rem}
Heuristically it is useful to consider that  the differential $d^i$ is induced by
the surjection of complexes $\kz _0 \twoheadrightarrow \sigma_{\geq i} \kz_0$
for $i \geq 1$, using Remark \ref{rem:E1_hz} to interpret the connecting morphism as $d^1$.
\end{rem}

\subsection{Bicomplexes of $S^\bullet$-modules}
\label{subsect:bicomplexes}

The purpose of this section is to explain the calculation of the local cohomology of $\torsion ku^* (BV_+)$; a fundamental point is that the method also calculates the local cohomology of $\cotorsion ku^* (BV) / v$ and  explains all the differentials in the local cohomology spectral sequence. This relies on the following result, in which grading shifts have been suppressed and, for variance reasons, the cohomology of $V^\sharp$ is considered.

\begin{prop}
(Cf. \cite[Section 4.6]{bruner_greenlees}.)
For $V \in \obj \fdvs$ of rank $r$, $ku^* (BV_+^\sharp)$ admits a finite natural filtration with associated graded
$$\mathrm{gr} \torsion ku^* (BV^\sharp_+)
\cong 
\bigoplus_{i=2} ^ r \lfrak_{i, *} (V)$$
 in $\smodaut$. Moreover,  as a module over $S^\bullet(V)$:
\[
\torsion ku^* (BV^\sharp_+) 
\cong 
\bigoplus_{i=2} ^ r \lfrak_{i, *} (V).
\]
\end{prop}

\begin{proof}
 The result follows from Lemma \ref{lem:associated_graded_L} and Theorem \ref{thm:ku_cohom_BV}.
\end{proof}

For the consideration of local duality, it is necessary to consider the
$S^\bullet$-action on the $\Lambda^a \otimes S^b $-bicomplex introduced in
Section \ref{subsect:filter_symm_powers}. This gives a half plane bicomplex with
differentials of the form
\[
	\xymatrix{
& \ar@{.>}[d] & \ar@{.>}[d]
\\
\ar@{.>}[r]
&
\Lambda^{j+2} \otimes S^{\bullet -3}
\ar[r]^{\tau_0}
\ar[d]_{\tau_1}
&
\Lambda^{j+1} \otimes S^{\bullet -2}
\ar[d]^{\tau_1}
\ar@{.>}[r]
&&&
\\
\ar@{.>}[r]
&
\Lambda^{j+1} \otimes S^{\bullet -1}
\ar[r]_{\tau_0}\ar@{.>}[d]
&
\Lambda^{j} \otimes S^{\bullet}
\ar@{.>}[r]\ar@{.>}[d]
& & &
\ar[l]^s \ar[u]_t\\
&&&&
}
\]
in the $(s,t)$-plane.

Consider the following brutal truncations, which are analogues of the truncated
Koszul complexes of Notation \ref{nota:truncated_koszul}.

\begin{defn} (Cf.  \cite[Section 4.6]{bruner_greenlees})
For $i\in \nat$, let
\begin{enumerate}
	\item 
$\bicom(i)$ be the bicomplex in $\smodf$:
\[
 \bicom(i)_{s,t} :=
\left\{
\begin{array}{ll}
 0 & t < i \ \mathrm{or} \ s <0;\\
\Lambda ^{s+t} \otimes S^\bullet & t \geq i,
\end{array}
\right.
\]
considered as a quotient bicomplex, where the term of lowest total degree is
$\Lambda^i \otimes S^\bullet$, in bidegree $(0,i)$.
\item
$\dbicom(i)$ be the bicomplex in $\smodf$:
\[
 \dbicom(i)_{s,t} :=
\left\{
\begin{array}{ll}
 0 & t > i \ \mathrm{or} \ s >0;\\
\Lambda ^{s+t} \otimes S^\bullet & t \leq i,
\end{array}
\right.
\]
considered as a sub-bicomplex, where the term of greatest total degree is
$\Lambda^i \otimes S^\bullet$, in bidegree $(0,i)$.
\end{enumerate}
\end{defn}

\begin{rem}
 When evaluated on $V \in \obj \fdvs$ of rank $r$, the only non-trivial terms
are those with $s +t \leq r$ and $t \geq i$. In particular, $\bicom(i)(V)$ is trivial if $i> r$.
\end{rem}

\begin{exam}
 Taking $r=5$ and $i=2$, $\bicom(2) (\field^5)$ is given by evaluating the
following
bicomplex on $\field^5$:
\[
 \xymatrix{
&&&\Lambda^5\otimes S^{*-6}
\ar[d]^{\tau_1}
\\
&& \Lambda^5\otimes S^{*-5}
\ar[d]^{\tau_1}
\ar[r]^{\tau_0}
&
\Lambda^4 \otimes S^{*-4}
\ar[d]^{\tau_1}
\\
&
\Lambda^5 \otimes S^{*-4}
\ar[d]^{\tau_1}
\ar[r]^{\tau_0}
&
\Lambda^4 \otimes S^{*-3}
\ar[d]^{\tau_1}
\ar[r]^{\tau_0}
& \Lambda^3  \otimes S^{*-2}
\ar[d]^{\tau_1}&&
\\
\Lambda^5 \otimes S^{*-3}
\ar[r]^{\tau_0}
&
\Lambda^4 \otimes S^{*-2}
\ar[r]^{\tau_0}
&
\Lambda^3  \otimes S^{*-1}
\ar[r]^{\tau_0}
&
\Lambda^2 \otimes S^{*} &
&
\ar[l]^s
\ar[u]_t
}
\]
where $\Lambda^2 \otimes S^*$ is in $(s,t)$-degree $(0,2)$.

Similarly, $\dbicom (2) (\field^5) $ is obtained by evaluating the following on
$\field^5$:
\[
	\xymatrix{
\Lambda^2 \otimes S^*
\ar[r]^{\tau_0}
\ar[d]_{\tau_1}
&
\Lambda^1 \otimes S^{*+1}
\ar[r]^{\tau_0}
\ar[d]_{\tau_1}
&
S^{*+2}
\\
\Lambda^1 \otimes S^{*+2}
\ar[r]^{\tau_0}
\ar[d]_{\tau_1}
&
\Lambda^1 \otimes S^{*+3}
\\
S^{*+4}.
	}
\]

\end{exam}

\begin{rem}
\ 
\begin{enumerate}
	\item
The grading on $\bicom (i)$ used in \cite[Chapter 4]{bruner_greenlees}
(respectively on $\dbicom (i)$)  can be recovered
by considering the grading of the `generators' in the lowest (resp. greatest)
total degree, since the morphism $\tau_0$ raises the degree by $2$ and $\tau_1 $
raises the degree by $4$ (the grading is
calculated relative to the $S^\bullet$-grading, so that the usual gradings of the odd degree
generators do not contribute).
	\item
	The homology of the bicomplexes is calculated pointwise,  by
first evaluating on $V\in \obj \fdvs$; for any $i$ and $V$,  the bicomplex $\bicom (i) (V)$ has only
finitely many non-zero terms. 
\end{enumerate}
\end{rem}

The following is clear from the definition.

\begin{lem}
\label{lem:filter-bicomplex}
\
\begin{enumerate}
	\item
	 There are  surjections of bicomplexes in $\smodf$:
\[
	\bicom (0)
	\twoheadrightarrow
	\bicom(1)
	\twoheadrightarrow
	\ldots
	\twoheadrightarrow
	 \bicom(i)
	 \twoheadrightarrow
	 \bicom(i+1)
	 \twoheadrightarrow
	 \ldots
\]
and,  the kernel of $\bicom(i)
	 \twoheadrightarrow
	 \bicom(i+1)$
	 is the truncated Koszul complex $\sigma_{\geq i} \kz_0$:
\[
\ldots \rightarrow  \Lambda^{i+s} \otimes S^{*-s} \stackrel{\tau_0}{\rightarrow}
\Lambda^{i+s-1} \otimes S^{*-s+1} \rightarrow \ldots \rightarrow \Lambda^i
\otimes S^*
\]
concentrated in $t$-degree $i$.
\item
There are inclusions of bicomplexes:
\[
	\dbicom (0) = S^\bullet \rightarrow \dbicom (1) \hookrightarrow \dbicom
(2) \hookrightarrow \ldots \hookrightarrow \dbicom (j-1) \hookrightarrow \dbicom (j)
\hookrightarrow \ldots
\]
and the cokernel of $\dbicom (j-1) \hookrightarrow \dbicom (j)$ is the truncated
Koszul complex $\sigma_{\leq j} \kz_0$ shifted so the term of maximal total
degree is in bidegree $(0,j)$.
\end{enumerate}
\end{lem}

The following result follows from the definition of the objects $\kfrak_{a,b}$ and $\lfrak_{a,b}$ given in Section \ref{sect:milnor}.

\begin{lem}
\label{lem:lfrak_smodf}
	For $0<a\in \zed$, $\kfrak_{a,\bullet}$ and  $\lfrak_{a,\bullet}$ are
objects of $\smodf$ such that  the surjections $\Lambda^a \otimes S^\bullet
\twoheadrightarrow \kfrak_{a,\bullet}\twoheadrightarrow \lfrak_{a, \bullet}$ are
morphisms of $\smodf$.
\end{lem}

To establish the behaviour of the indices, consider the following commutative diagram (for $a \geq 2$):
\[
 \xymatrix{
\Lambda^{a+1} \otimes S^{\bullet -2} 
\ar[d]_{\tau_1}
\ar@{->>}[r]
&
\kfrak_{a+1, \bullet -2}
\ar[d]
\\
\Lambda^a \otimes S^\bullet 
\ar@{->>}[r]
\ar[dd]_{\tau_1}
&
\kfrak_{a,\bullet}
\ar@{->>}[d]
\\
&
\lfrak_{a, \bullet}
\ar@{^(->}[d]
\\
\Lambda^{a-1} \otimes S^{\bullet+2} 
\ar@{->>}[r] 
\ar@/_2pc/[rr]^{\tau_0}
& \kfrak_{a-1, \bullet +2}
\ar@{^(->}[r]
&
\Lambda^{a-2} \otimes S^{\bullet+3} .
}
\]
In particular, $\lfrak_{a, b}$ is a quotient of $\Lambda^a \otimes S^b $ and a subobject of $\Lambda^{a-2} \otimes S^{b+3}$.

In the following Proposition, recall the indexing convention of $\bicom(i)$, 
which
means that the term of lowest total degree is in $(s,t)$-degree $(0,i)$, hence
has total degree $i$.

\begin{prop}
\label{prop:hom_bicom}
 For $0< i\in \zed$, the homology of $\mathrm{Tot} (\bicom(i))$ is
concentrated in degree $i$ and $
 H_i (\mathrm{Tot} (\bicom(i))) \cong
 \lfrak_{i, *}.
$
\end{prop}

\begin{proof}
 The result follows by filtering the bicomplex $\bicom(i)$ using Lemma
\ref{lem:filter-bicomplex} and Proposition \ref{prop:J-complex_homology}
(cf. \cite[Proposition 4.6.3]{bruner_greenlees}).
\end{proof}

\begin{rem}
	Although the bicomplex $\bicom(0)$ is defined, it does not give a
resolution of $\field$, since the homology is not concentrated in
a single homological degree - which is a consequence of the homology of
$(\Lambda^\bullet \otimes S^\bullet, \tau_0)$. 
\end{rem}

\begin{prop}
\label{prop:hom_dbicom}
	For  $i \in \zed$, the homology of $\mathrm{Tot} (\dbicom(i))$ is as follows. 

For $i=0$:
$$
 H_m (\mathrm{Tot} (\dbicom(0)))\cong  
\left\{
\begin{array}{ll}
  S^\bullet & m=0\\
0 & \mathrm{otherwise}.
\end{array}
\right.
$$

For $i >0$: 
\begin{eqnarray*}
	H_m (\mathrm{Tot} (\dbicom(i))) &\cong &
\left\{
\begin{array}{ll}
\field ^{\oplus {i+1}} \oplus \bigoplus_{d\geq 1}p_d \Ibar & m =0 \\
 \lfrak_{i+2, *} & m=i\\
0& \mathrm{otherwise}.
\end{array}
\right.
\end{eqnarray*}

Moreover, for $i >1$, the short exact sequence of bicomplexes $0 \rightarrow \dbicom (i-1) \rightarrow \dbicom (i) \rightarrow \sigma_{\leq i} \kz_0 \rightarrow 0$ given by Lemma \ref{lem:filter-bicomplex}, induces the short exact sequences 
\[
 \xymatrix{
H_0 (\mathrm{Tot} (\dbicom(i-1)))
\ar[r]
\ar[d]_\cong 
&
H_0 (\mathrm{Tot} (\dbicom(i)))
\ar[r]
\ar[d]_\cong 
&
H_0 (\sigma_{\leq i}\kz_0)
\ar[d]^\cong 
\\
\field ^{\oplus {i}} \oplus \bigoplus_{d\geq 1}p_d \Ibar
\ar[r]
&
\field ^{\oplus {i+1}} \oplus \bigoplus_{d\geq 1}p_d \Ibar
\ar[r]
&
\field
}
\]
and 
\[
\xymatrix{
H_i (\mathrm{Tot} (\dbicom(i)))
\ar[r]
\ar[d]_\cong 
&
H_i (\mathrm{Tot} (\sigma_{\leq i}\kz_0)) 
\ar[r]
\ar[d]_\cong 
&
H_{i-1} (\mathrm{Tot} (\dbicom(i-1))) 
\ar[d]^\cong 
\\
\lfrak_{i+2, *}
\ar[r]
&
\kfrak_{i+1, *}
\ar[r]
&
\lfrak_{i+1, *}
} 
\]
in homology.
\end{prop}

\begin{proof}
	The calculation of the homology follows from Proposition \ref{prop:J-complex_homology},
together with the fact that each row of $\dbicom(i)$, considered as a truncated Koszul complex, 
contributes a factor $\field$ in homological degree $0$.

For $i >1$, the given short exact sequences follow immediately; the exactness of the second sequence is again a consequence of Proposition \ref{prop:J-complex_homology}.
\end{proof}

\begin{rem}
\ 
\begin{enumerate}
 \item 
The factors $\field$ (resp. the different factors $p_d \Ibar$) lie in distinct gradings.
\item
The degree $r-1$ homology of $\dbicom (r-1)$ is the functor $\lfrak_{r+1, *}$, which is a quotient of $\Lambda^{r+1} \otimes S^\bullet$. In particular, when evaluated on the rank $r$ space $V$, this is trivial. Thus, the homology of $\dbicom (r-1)$ evaluated on $V$ is concentrated in degree zero.
\end{enumerate}
\end{rem}

The following result is proved by using the identification of local cohomology given by Proposition \ref{prop:local_cohomology_functors}.  Here $D$ denotes the (graded) duality functor and the identification $D p_d \Ibar  \cong q_d \Pbar$ is used to give the duality statement $ D \lfrak_{1,*} \cong \bigoplus_{d>0} q_d \Pbar $; all functors should be understood as being evaluated on $V$.

\begin{cor}
\label{cor:local_cohom_lfrak}
 (Cf. \cite[Theorem 4.7.3]{bruner_greenlees}.)
For $V \in \obj \fdvs$ of rank $r >1$, the local cohomology of $\lfrak_{i,*}$ 
\begin{enumerate}
 \item 
for $i=1$,  is concentrated in cohomological degree one and
 $
 H^1 _ I (\lfrak_{1,*} ) \cong \field ^{\oplus r} \oplus D \lfrak_{1,*};
$
\item
for $2 \leq i \leq r-1$:
\begin{eqnarray*}
 H^m _I (\lfrak_{i, *} ) \cong 
\left\{
\begin{array}{ll}
 \field^{\oplus r-i +1 } \oplus  D \lfrak_{1,*}& m=i \\
 D \lfrak _{r-i+2, *} & m=r \\
0 & \mathrm{otherwise};
\end{array}
\right.
\end{eqnarray*}
\item
for $i=r$,  is concentrated in cohomological degree $r$ and
$
 H^r _I (\lfrak_{r,*} ) \cong D S^\bullet.
$
\end{enumerate}
Moreover, the surjection of complexes $\bicom (1) \twoheadrightarrow \bicom (i)$, for $1 < i < r$, induces a surjection:
 $
 H^1 _I (\lfrak_{1,*} ) \twoheadrightarrow H^i_I (\lfrak_{i,*}) 
$ 
with kernel $\field^{\oplus i-1}$.
\end{cor}

\begin{rem}
 \ 
\begin{enumerate}
 \item 
The restriction $r >1$ on the rank is imposed so as to  give a unified statement.
\item
The grading is again suppressed; the reader is encouraged to calculate the appropriate gradings and to verify that the above
 yields the Hilbert series specified in \cite[Section 4.7]{bruner_greenlees}.
\item
A notational sleight of hand has been used; the local cohomology introduces a change of variance with respect to $V$; $\lfrak_{i,*}$ on the left hand side  should be considered as being {\em contravariant} in $V$, by precomposition with the duality functor $\sharp$.
\end{enumerate}
\end{rem}

\subsection{The local cohomology of $\cotorsion ku^* (BV_+)$}
\label{subsect:cotorsion_loc_cohom}

Throughout this section, $V \in \obj \fdvs$ has rank $r>1$. To simplify notation, $Q$ will be written for the (contravariant) functor $V \mapsto \cotorsion ku^* (BV_+) $. Multiplication by $v$ induces a short exact sequence 
\[
 0 
\rightarrow 
Q \stackrel{v}{\rightarrow} 
Q \rightarrow Q/v
\rightarrow 0
\]
(omitting the suspension associated with the grading). Moreover, the augmentation $Q \rightarrow \zed [v]$ induces a short exact sequence
\[
 0 
\rightarrow 
\overline{Q} /v \rightarrow Q/v \rightarrow \zed 
\rightarrow 
0.
\]

Theorem \ref{thm:ku_cohom_BV} implies that there is a natural isomorphism 
$
 \overline{Q} /v 
\cong 
\lfrak_{1, *} (V^\sharp).
$
 Hence, by using the change of rings isomorphism associated to $ku^* (BV_+) \twoheadrightarrow \hz^* (BV_+)$, Corollary \ref{cor:local_cohom_lfrak} gives the local cohomology of $\overline{Q}/v$ ; in particular, the ideal $I$ refers here to the augmentation ideal of $ku^* (BV_+)$.

\begin{prop}
\label{prop:loc_cohom_Q/v}
 The local cohomology of $Q/v$ is concentrated in cohomological degrees zero and one:
\begin{eqnarray*}
 H^0 _I (Q/v) & \cong & 2 \zed 
\\
H^1 _I (Q/v) &\cong & \field^{\oplus r-1} \oplus D \lfrak_{1,*}.
\end{eqnarray*}
In particular, $2 H^1 _I (Q/v) =0$.
\end{prop}

\begin{proof}
 The result follows from the exact sequence 
\[
 0 
\rightarrow
H^0 _I (Q/v)
\rightarrow 
\zed 
\rightarrow 
H^1_I (\overline{Q}/v) 
\rightarrow 
H^1_I (Q/v)
\rightarrow 
0. 
\]
The connecting morphism is non-trivial (this can be seen by considering the behaviour modulo $2$), whence the result. (This also explains the notation $2\zed$).
\end{proof}

The local cohomology of $Q$ can now be analysed by using the exact sequence associated to $Q \stackrel{v}{\rightarrow} Q \rightarrow Q/v$, which has the form:
\[
 0 
\rightarrow 
H^0_I (Q) 
\stackrel{v}{\rightarrow}
 H^0_I (Q)
\rightarrow 
H ^0 _I(Q/v) \cong 2 \zed
\rightarrow 
H^1_I (Q)
\stackrel{v}{\rightarrow}
H^1_I (Q)
\rightarrow 
H^1_I (Q/v)
\rightarrow 
0.
\]
The calculation of $H^0 _I (Q) $ is straightforward (cf. \cite[Section 4.4]{bruner_greenlees}) and this implies that the image of the connecting morphism is $\zed/2^{r-1}$. This is sufficient to calculate the local cohomology. A direct approach is taken in \cite{bruner_greenlees}; the above is preferred here since it stresses the relationship between $H^1_I (Q/v)$ and the $v$-adic filtration of $H^1_I (Q)$.

\begin{prop}
\label{prop:v-adic_H1}
(Cf. \cite[Lemma 4.5.1]{bruner_greenlees})
 The associated graded of the $v$-adic filtration of $H^1_I (Q)$ is given by 
\[
 v^i H^1 _I (Q) / v^{i+1} H^1_I Q \cong 
\left\{
\begin{array}{ll}
 \big( \field ^{\oplus r-i-1} \oplus D \lfrak_{1,*} \big ) \cong H^{i+2}_ I (\lfrak_{i+2,*}) & 0 \leq i \leq r-3 \\
\field \oplus D \lfrak_{1,*} & i = r -2;\\
D \lfrak_{1, *} & i \geq r-1.
\end{array}
\right.
\]
\end{prop}

\begin{proof}
 To prove that the naturality with respect to $V$ is correct, use the description $Q [\frac{1}{y^*}]/ Q$ which is given in \cite[Proposition 4.4.7]{bruner_greenlees}.
\end{proof}

\begin{rem}
 \ 
\begin{enumerate}
 \item 
Grading shifts are suppressed.
\item
 \cite[Lemma 4.5.1]{bruner_greenlees} is a statement about the $2$-adic filtration; this coincides with the $v$-adic filtration (compare the final statement of Lemma \ref{lem:Q_basis_properties}). 
\item
The distinction between the cases $i \leq r-3$ and $i = r-2$ is simply to underline the isomorphism with $H^{i+2}_I (\lfrak_{i+2,*})$; this should be related to the surjection between local cohomology groups given in Corollary \ref{cor:local_cohom_lfrak}.
\end{enumerate}
\end{rem}

\subsection{The local cohomology spectral sequence for $ku_*(BV_+)$}

The local cohomology of $ku^* (BV_+)$ is determined by using the short exact sequence
\[
 0 
\rightarrow 
\torsion ku^* (BV_+) 
\rightarrow 
ku^* (BV_+ ) 
\rightarrow 
\cotorsion ku^* (BV_+)
\rightarrow 
0, 
\]
and the calculation of the local cohomology of $\torsion ku^* (BV_+)$, which follows from the results of Section \ref{subsect:bicomplexes}, and of $\cotorsion ku^* (BV_+)$, given in Section \ref{subsect:cotorsion_loc_cohom}.

\begin{rem}
As in \cite[Section 4.11]{bruner_greenlees}, it is more transparent to consider an $E^1$-page of the local cohomology spectral sequence which serves to calculate the local cohomology at the $E^2$-page. 
\end{rem}

\begin{thm}
 \label{thm:local_cohom_ss}
(Cf.  \cite[Section 4.11]{bruner_greenlees}.) 
For $V \in \obj \fdvs$ of rank $r>1$, 
there is a local cohomology spectral sequence converging to  $ku^* (BV_+)$ with (up to shift in grading):
\[
 E^1_{s, *} = 
\left\{
\begin{array}{ll}
 \zed [v] & s = 0
\\
H^1_I (Q) & s =-1
\\
H^{-s}_I (\lfrak_{-s,*} ) & -2 \geq s \geq 1-r
\\
H^r _I ( ku^* (BV_+) ) & s= -r, 
\end{array}
\right.
\]
where  $H^r _I (ku^* (BV_+) ) $ fits into an exact sequence 
\[
 0 
\rightarrow 
\field \oplus D \lfrak_{1,*} \rightarrow H^r _I (ku^* (BV_+) )
\rightarrow 
H^r _I (\hz^* (BV_+) ) 
\]
under the morphism induced by $ku^* (BV_+) \rightarrow \hz^* (BV_+)$. 

The $E^\infty$-page is given by
\[
 E^\infty_{s, *} = 
\left\{
\begin{array}{ll}
\zed [v] & s = 0
\\
v^{r-1} H^1_I (Q) & s =-1
\\
0 & -2 \geq s \geq 1-r 
\\
\mathrm{Image} \{ H^r _I ( ku^* (BV_+) )\rightarrow H^r _I (\hz^* (BV_+) ) \} & s= -r.
\end{array}
\right.
\]
Moreover, the morphism $H^0_I (ku^* (BV_+)) / v \rightarrow H^0_I (\hz^* (BV_+))$ is a bijection onto the permanent cycles in the zero column of the local cohomology spectral sequence for $\hz_* (BV_+)$.
\end{thm}

\begin{proof}
 The $E^1$-page of the spectral sequence is determined by combining the results of  Corollary \ref{cor:local_cohom_lfrak}, for the contribution from the local cohomology of $\torsion ku^* (BV_+)$, and of Section \ref{subsect:cotorsion_loc_cohom} for the local cohomology of $Q$. The results of Section  \ref{subsect:bicomplexes} provide the exact sequence for $H^r _I (ku^* (BV_+) ) $.

The entire behaviour of the spectral sequence is determined  by the fact that the only non-trivial differentials originate on the $s=-1$ column, using  Proposition \ref{prop:v-adic_H1} to relate this to the $v$-adic filtration of $H^1_I (Q)$. 
\end{proof}

\begin{rem}
Corollary \ref{cor:local_cohom_lfrak} suggests the heuristic explanation that the differentials are induced by the surjection of bicomplexes $\bicom (1) \twoheadrightarrow \bicom (i)$, for $1 < i < r$. The kernel of the map induced in local cohomology has already been accounted for in the calculation of $H^1_I (Q)$.
\end{rem}


\providecommand{\bysame}{\leavevmode\hbox to3em{\hrulefill}\thinspace}
\providecommand{\MR}{\relax\ifhmode\unskip\space\fi MR }
\providecommand{\MRhref}[2]{%
  \href{http://www.ams.org/mathscinet-getitem?mr=#1}{#2}
}
\providecommand{\href}[2]{#2}

\end{document}